\documentclass[11pt,twoside]{amsart}

\usepackage{amscd,amsfonts,amsmath,amssymb,amsthm,anyfontsize,bbding,bigints,blindtext,centernot,cite,colortbl,enumitem,extarrows,fancyhdr,float,footmisc,galois,geometry,graphicx,hyperref,imakeidx,lastpage,letltxmacro,lipsum,lmodern,marginnote,mathrsfs,mathtools,pgfplots,spectralsequences,stmaryrd,subcaption,svg,thmtools,thm-restate,tikz,tikz-cd,titlecaps,xcolor}
\usepackage[utf8]{inputenc}
\usepackage[T1]{fontenc}

\hypersetup{
    colorlinks,
    linkcolor={red!50!black},
    citecolor={blue!70!black},
    urlcolor={green!80!black}}

\makeatletter
\def\blfootnote{\xdef\@thefnmark{}\@footnotetext}
\makeatother

\usetikzlibrary{
    positioning,
    decorations.markings,
    decorations.pathreplacing,
} 
\usetikzlibrary{patterns}
\pgfplotsset{compat=1.16}
\usetikzlibrary{calc}

\newtheorem{theorem}{Theorem}[section]
\newtheorem{prop}[theorem]{Proposition}
\newtheorem{definition}[theorem]{Definition}
\newtheorem{corollary}[theorem]{Corollary}
\newtheorem{lemma}[theorem]{Lemma}

\newtheorem{example}[theorem]{Example}

\newtheorem{construction}[theorem]{Construction}
\newtheorem{claim}[theorem]{Claim}

\newtheorem{remark}[theorem]{Remark}

\counterwithin{figure}{subsection}

\newcommand{\R}{\mathbb{R}}
\newcommand{\C}{\mathbb C}
\newcommand{\Z}{\mathbb{Z}}

\newcommand{\T}{\mathbb{T}}

\newcommand{\Q}{\mathbb{Q}}

\newcommand{\hyp}{\mathbb{H}}
\newcommand{\sph}{\mathbb{S}}
\newcommand{\QA}{\mathbb{A}}
\newcommand{\SL}{\mathrm{SL}}
\newcommand{\GL}{\mathrm{GL}}
\newcommand{\M}{\mathrm{M}}
\newcommand{\id}{\mathrm{id}}
\newcommand{\PSL}{\mathrm{PSL}}
\newcommand{\SU}{\mathrm{SU}}

\newcommand{\Id}{\text{Id}}

\newcommand{\calO}{\mathcal O}
\newcommand{\Padic}{\mathcal P}

\makeatletter
\newcommand*\rel@kern[1]{\kern#1\dimexpr\macc@kerna}
\newcommand*\widebar[1]{%
  \begingroup
  \def\mathaccent##1##2{%
    \rel@kern{0.8}%
    \overline{\rel@kern{-0.8}\macc@nucleus\rel@kern{0.2}}%
    \rel@kern{-0.2}%
  }%
  \macc@depth\@ne
  \let\math@bgroup\@empty \let\math@egroup\macc@set@skewchar
  \mathsurround\z@ \frozen@everymath{\mathgroup\macc@group\relax}%
  \macc@set@skewchar\relax
  \let\mathaccentV\macc@nested@a
  \macc@nested@a\relax111{#1}%
  \endgroup
}
\makeatother

\newcommand{\inv}{{-1}}
\newcommand{\Ram}{\text{Ram}}

\usepackage[dvipsnames]{xcolor}

\let\oldmarginnote\marginnote
\renewcommand{\marginnote}[1]{%
  \oldmarginnote{\fcolorbox{red}{white}{\parbox{\dimexpr\marginparwidth-2\fboxsep-2\fboxrule\relax}{\footnotesize #1}}}%
}

\title{Tomter's example revisited}
\author{Danyu Zhang}
\address{Department of Mathematics, University of Luxembourg, Maison du Nombre, 6, avenue de la Fonte, L-4364 Esch-sur-Alzette, Luxembourg}
\email{danyu.zhang@uni.lu}

\begin{document}

\begin{abstract}
    We construct examples of fibrewise Anosov flows over the Fenley--Mann--Potrie example of dimension 4. This is a family of non-algebraic Anosov flows with unstable dimension $1+4n$ and stable dimension $2+4n$, for every positive integer $n$.
\end{abstract}

\maketitle
\thispagestyle{empty}

\bigskip

\section{Introduction}

A $C^1$ flow $\varphi^t:M\to M$ on a closed manifold $M$ is Anosov if there exists an invariant splitting $TM=E^s\oplus \R X\oplus E^u$, a Riemannian metric $\|\cdot\|$, and constants $C,\lambda>0$ such that for any $t\geq 0$, we have 
$$\|D\varphi^tv \|\leq Ce^{-\lambda t}\|v\|,\ \text{if } v\in E^s,\ \text{and} $$
$$\|D\varphi^{-t}v \|\leq Ce^{-\lambda t}\|v\|,\ \text{if } v\in E^u.$$

Anosov flows have been extensively studied since the 1960s, and their classification has been one of the central questions in the theory. The first examples are suspension flows of Anosov diffeomorphisms and geodesic flows of negatively curved manifolds. Tomter \cite{tomter} classified Anosov flows on closed homogeneous manifolds that are induced by actions of one-parameter subgroups. Such flows are either suspensions of Anosov diffeomorphisms of nilmanifolds, geodesic flows of hyperbolic manifolds, or flows on nilmanifold bundles fibering over geodesic flows of hyperbolic manifolds (nil-suspensions); see Theorem \ref{theorem: classification of algebraic flows}.

We are primarily interested in Anosov flows up to orbit equivalence. We call an Anosov flow \emph{algebraic} if it is orbit equivalent to one of the flows in Tomter's classification.

In dimension three, in addition to suspensions and geodesic flows, there are many non-algebraic examples of Anosov flows \cite{franks-williams, handel-thurston, bonatti-langevin, BBY, neige}. In particular, rich families of Anosov flows have been constructed on hyperbolic 3-manifolds \cite{goodman, bowden-mann, bonatti-iakovoglou, beguin-yu, BBMY}. In higher dimensions, however, non-algebraic examples remain scarce. To date, only two families of such flows are known, constructed in \cite{BBGRH} and \cite{fenley-mann-potrie}.

The presence of lower-dimensional Anosov flows naturally suggests considering bundle constructions. In \cite{BBGRH}, the authors introduced a family of \emph{fibrewise Anosov flows}: flows on torus bundles that are hyperbolic along the fibres and fibre over lower-dimensional Anosov flows (see Definition \ref{def: fibrewise Anosov}). However, no concrete examples were given.

Inspired by Tomter's algebraic examples of the third ``mixed type'', we aim to construct fibrewise Anosov flows as follows.

\begin{construction}\label{construction}
    Suppose we have an Anosov flow $\varphi^t: M\to M$.
    Let $\pi: \widetilde M\to M$ be a normal covering with the deck group $\Gamma:=\pi_1M/\pi_1\widetilde M$. 
    Suppose $\varphi^t$ lifts to $\tilde\varphi^t:\widetilde M\to \widetilde M$.
    We want to find a representation $\rho:\Gamma\to\GL(n,\Z)$. Let $\Gamma$ act on $\widetilde M\times \T^n$ by $\gamma.(x,y)=(\gamma.x, \rho(\gamma)y)$, where $(x,y)\in\widetilde M\times\T^n$ and $\gamma\in\Gamma$. 
    Define $\tilde\Phi^t(x,y)=(\tilde\varphi^t(x),y)$. The flow $\tilde\Phi^t$ projects to $\Phi^t:N\to N$ where $N:=\Gamma\backslash(\widetilde M\times\T^n)$, which is our natural candidate for a fibrewise Anosov flow.
\end{construction}

\begin{sloppypar}
Fenley--Mann--Potrie \cite{fenley-mann-potrie} constructed a family of codimension-one Anosov flows on $4$-manifolds that are $\sph^1$-bundles over hyperbolic $3$-manifolds. In \cite{tomter}, Tomter described actions of arithmetic Fuchsian groups on tori and used them to construct algebraic fibrewise Anosov flows over unit tangent bundles of surfaces.
\end{sloppypar}

This naturally led us to consider arithmetic Kleinian groups, which also admit actions on tori and provide the representations $\rho$ required in the construction above. We thus obtain the following theorem.

\begin{theorem} \label{theorem: main theorem}
    Let $M$ be a closed hyperbolic 3-manifold and $\widetilde M$ be its universal cover. Assume \begin{enumerate}
        \item the fundamental group of $M$ is an arithmetic Kleinian group derived from a quaternion division algebra over a number field with exactly one complex place and no real places;
        \item there exists an Anosov flow on the flat bundle $M':=\Gamma\backslash(\widetilde M\times F)$ where $F$ is a closed manifold and $\Gamma$ acts diagonally, and the flow lifts to $\widetilde M\times F$;
        \item the Anosov flow on $M'$ is uniformly quasi-geodesic with respect to the metric on $\widetilde M$.
    \end{enumerate}
    Then there exists a fibrewise Anosov flow over $M'$.
\end{theorem}

The construction relies on arithmetic Kleinian groups acting by left multiplication on $\M_2(\C)$, which as a byproduct gives $4$-Anosov representations. See Remarks \ref{remark: anosov repn via GW} and \ref{remark: anosov repn via singular values}.

While we do not exclude the possibility that $M'$ itself is a hyperbolic $3$-manifold, finding a base flow on a hyperbolic 3-manifold that satisfies all of the above conditions is beyond the scope of this paper. We include a brief discussion of this question at the end of the paper.

Since the Fenley--Mann--Potrie example \cite{fenley-mann-potrie} satisfies all the assumptions of Theorem \ref{theorem: main theorem}, we obtain the following as a corollary.

\begin{corollary} \label{coro: over FMP}
    There is a non-algebraic Anosov flow in dimension $4+8n$, $n\in\Z_{>0}$ that fibres over the Fenley--Mann--Potrie example \cite{fenley-mann-potrie}.
\end{corollary}

Verjovsky conjectured that every codimension-one Anosov flow is orbit equivalent to a suspension flow. The examples of Fenley--Mann--Potrie \cite{fenley-mann-potrie} provide a family of counterexamples to this conjecture. In \cite{BBGRH}, the authors posed a more general question: whether every Anosov flow with $\dim E^s\neq \dim E^u$ must be orbit equivalent to a suspension flow. 
Note that this statement is equivalent to the Verjovsky conjecture in dimensions 4 and 5, but has a weaker hypothesis (and hence a stronger conclusion) in dimensions $>5$.
Corollary \ref{coro: over FMP} provides a family of non-algebraic Anosov flows with unequal stable and unstable dimensions in dimensions $4+8n$, for every $n\in\Z_{>0}$.

Note that any non-algebraic Anosov flow in even dimension will provide a counterexample to the generalized Verjovsky conjecture. This naturally raises the question of whether non-algebraic Anosov flows exist in every even dimension.

Filip--Fisher--Lowe \cite{filip-fisher-lowe} showed that if a closed real-analytic Riemannian manifold, with negative sectional curvature and of dimension $n\geq 3$, contains infinitely many closed totally geodesic immersed hypersurfaces, then it is isometric to a hyperbolic manifold.
At the end of the same paper, a question was briefly mentioned on whether an Anosov flow that preserves infinitely many flow invariant submanifolds is necessarily algebraic. Zeghib \cite{zeghib} proved that the closure of a $C^1$ immersed sub-manifold of finite volume that is invariant under an algebraic Anosov flow is always ``semi sub algebraic'', i.e., it is a finite union of sets each obtained by flowing the projection of a coset of a closed subgroup under a one parameter subgroup.
In this context, Corollary \ref{coro: over FMP} produces an example of a non-algebraic Anosov flow that has infinitely many flow invariant submanifolds; see Remark \ref{remark: inf many flow inv submanifolds}. 

We are going to state some definitions and theorems in Section \ref{section: preliminaries}, construct the actions of arithmetic Kleinian groups on tori in Section \ref{section: tomter's construction}, and prove Theorem \ref{theorem: main theorem} and Corollary \ref{coro: over FMP} to obtain concrete examples in Section \ref{section: proof of the main theorem}.

\medskip

\noindent\textbf{Acknowledgements.} I would like to thank Rafael Potrie for the opportunity to visit him in Montevideo, telling me about the 4D example, suggesting the question, and many communications; Claudio Llosa Isenrich for reading and for his quick and helpful feedback, especially on the first three sections; and Adam Chalumeau for discussions about $\pi_1$.
I am still grateful to Andreas Wieser and Osama Khalil who pointed me to arithmetic Fuchsian groups in 2022 during the Margulis conference (though I should have learned about them better from Hannah Alpert's GGT class).

\medskip

\noindent\textbf{AI use.}
I maintained conversations with ChatGPT (GPT-5.5 at the time) throughout this project, and used it for grammar and spelling checks, but not math errors prior to v1.
Most importantly, it helped me read Maclachlan--Reid much faster, I identified the correct conditions for quaternion algebras during these conversations, and it generated the concrete Example \ref{example: quaternion algebra} of $\left(\frac{i+1,3}{\Q(i)}\right)$ which in turn helped understanding. Prior to v2, I used GPT-5.6 Sol for the scanning of math errors globally, and it caught that the arguments on how to extend the uniform hyperbolicity from return maps over periodic orbits need to be further justified, with which I agree. It suggested a fix, and this is the approach that I am taking in the current v2, after bibliography checks and chatting with human mathematicians. 
I do not ask AI to directly generate any texts or proofs.


\bigskip

\section{Preliminaries} \label{section: preliminaries}

\subsection{Anosov flows and quasi-geodesics}

\begin{definition} \label{def: fibrewise Anosov}
    Let $\Phi^t: E\to E$ be a $C^1$ flow, where $\begin{tikzcd}[column sep=small]
        F\rar & E\rar["\pi"] & B
    \end{tikzcd}$ is a fibre bundle. We say that $\Phi^t$ is \emph{fibrewise Anosov}, if $\pi\circ\Phi^t=\varphi^t\circ\pi$ for some $C^1$ Anosov flow $\varphi^t$ and there exist $d\Phi^t$-invariant subbundles $TF=V^s\oplus V^u$, and constants $C>0,\lambda\in(0,1)$ such that for $t>0$
    $$\|d\Phi^tv\|\leq C\lambda^t\|v\|,\ v\in V^s,\ \ \ \text{and}\ \ \ \|d\Phi^{-t}v\|\leq C\lambda^t\|v\|,\ v\in V^u.$$
\end{definition}
Such flows are Anosov flows on the total space by \cite{BBGRH}. Sometimes we define a fibrewise Anosov flow to be only hyperbolic in the fibres, over an arbitrary flow on the base, but in this paper, we are only interested in flows that are Anosov on the total space.

We say that two flows $\varphi^t:M\to M$ and $\psi^t:M\to M$ are \emph{orbit equivalent} if there exists a homeomorphism, sending orbits of $\varphi^t$ to orbits of $\psi^t$ preserving or reversing the orientation of orbits consistently, but not necessarily preserving the parametrization.

We state Tomter's classification theorem for algebraic Anosov flows \cite{tomter} here. Let our manifold be $M:=\Gamma\backslash G/K$, where $G$ is a Lie group with Lie algebra $\mathfrak g$, $K$ is a compact subgroup such that $G/K$ is connected, and $\Gamma<G$ is a uniform lattice. The flow is defined by $\varphi^t:\Gamma gK\mapsto \Gamma g\exp(t\alpha) K$, for some $\alpha\in\mathfrak g$. We assume $\exp(t\alpha)$ commutes with $K$, $\varphi^t$ commutes with $\Gamma$, and $\text{ad}\alpha$ has no nonzero imaginary eigenvalues. We call $\varphi^t$ a \emph{$(G,\Gamma)$-induced Anosov flow}.
Intuitively, the classification theorem follows from the Levi decomposition of $\mathfrak g$.
More generally, we say that $\varphi^t$ is an \emph{algebraic Anosov flow} if it is orbit equivalent to a $(G,\Gamma)$-induced Anosov flow. See also \cite{barbot-maquera}.

\begin{theorem}[Classification of algebraic Anosov flows] \label{theorem: classification of algebraic flows}
    Suppose $\varphi^t$ is an algebraic Anosov flow. Then up to orbit equivalence, it is either a suspension flow of an Anosov diffeomorphism of a nilmanifold, a geodesic flow of a hyperbolic manifold, or a fibrewise Anosov flow on a nilmanifold bundle fibering over a geodesic flow of a hyperbolic manifold.
\end{theorem}

The following notion will be convenient in our writing.

\begin{definition}
    Let $\Gamma$ be a discrete group, and $\rho:\Gamma \to \SL(d,\Z)$ be a representation. Let $\Gamma$ act on the torus $\T^d$ by $\gamma.x=\rho(\gamma)x$ for all $\gamma\in \Gamma$ and $x\in\T^d$. We say that the action of $\Gamma$ on $\T^d$ or $\rho$ is \emph{totally Anosov}, if $\rho(\gamma)$ is an Anosov toral automorphism for every $\id\neq \gamma\in\Gamma$.
\end{definition}

\noindent \textbf{Convention:} For Construction \ref{construction}, we are always going to use the convention of a left action, i.e., suppose $\Gamma$ is a group that acts on $X\times Y$ where $X,Y$ are topological spaces, and we always write that $\gamma\in\Gamma$ acts on $X,Y$ both on the left and $\gamma.(X,Y)=(\gamma.X,\gamma. Y)$.

\medskip

We use the property that the base flow is quasi-geodesic to obtain estimates in the fibres.

\begin{definition}
    Let $(X,d_X)$ and $(Y,d_Y)$ be two metric spaces. We say $\phi:X\to Y$ is a \emph{$k$ quasi-isometric embedding} if there is a constant $k\geq 1$ such that for any $p,q\in X$,
    $$\frac{1}{k}d_Y(\phi(p),\phi(q))-k\leq d_X(p,q)\leq kd_Y(\phi(p),\phi(q))+k.$$
    If $(X,d_X)$ is isometric to $(\R,|\cdot|)$, then the image of $X$ under a $k$ quasi-isometric embedding is a \emph{$k$ quasi-geodesic}.
\end{definition}

\begin{definition}
    A $C^1$ flow $\varphi^t$ on a 3-manifold is \emph{quasi-geodesic} if each flowline $\ell$ of $\varphi^t$ with its induced length metric is a quasi-geodesic in $\widetilde M$. A flow $\varphi^t$ is \emph{uniformly quasi-geodesic} if there is a constant $k$ such that every flowline $\ell$ of $\varphi^t$ is a $k$ quasi-geodesic.
\end{definition}

We define the following slightly more general notion of a quasi-geodesic flow.

\begin{definition}
    Let $\varphi^t$ be a flow on a flat $F$-bundle $M'$ over a 3-manifold $M$ where the fibre $F$ is a closed manifold, and suppose it lifts to a flow $\tilde\varphi^t$ on $\widetilde M\times F$. Let $\pi:\widetilde M\times F\to\widetilde M$ denote the projection to the first coordinate. We say $\varphi^t$ is \emph{uniformly quasi-geodesic with respect to the metric on $\widetilde M$}, if there exists $k\geq 1$ such that for two points $p,q$ that belong to any flowline $\ell$ in $\widetilde M\times F$, we have 
    $$\frac{1}{k}d_{\ell}(p,q)-k\leq d_{\widetilde M}(\pi(p),\pi(q)) \leq kd_{\ell}(p,q)+k,$$
    where $d_\ell$ denotes the distance along $\ell$ induced from the metric on $M'$.
\end{definition}

\subsection{Fields and quaternion algebras} \label{subsection: fields and quaternion algebras} We recall some algebra. Most of the theory in Sections \ref{subsection: fields}, \ref{subsection: quaternion alg} and \ref{subsection: arith kleinian groups} can be found in Maclachlan--Reid \cite{maclachlan-reid}, which we include for completeness.

\subsubsection{Number fields and $\Padic$-adic fields} \label{subsection: fields}
We recall some field theory. In this section, we will always denote a general field by $K$ and a number field by $k$.

Recall that a number field $k$ is always a simple extension of $\Q$, and the embedding $k\hookrightarrow\C$ is determined by where it sends the primitive element.
Suppose $k=\Q(\alpha)$ and the minimal polynomial of $\alpha$ has $r_1$ real roots and $r_2$ pairs (counting pairs of complex conjugates) of complex roots. There are $n=r_1+2r_2$ embeddings, where $n$ is exactly the degree of the field extension.

A \emph{(multiplicative) valuation} on a field $K$ is a function $v: K\to\R,\ x\mapsto v(x)$ such that 
\begin{center}
    (1) $v(x)>0$ except $v(0)=0$;\quad (2) $v(xy)=v(x)v(y)$;\quad  (3) $v(x+y)\leq v(x)+v(y)$.
\end{center} If the stronger condition
\begin{center}
    (3') $v(x+y)\leq\max\{v(x),v(y)\}$
\end{center} holds, then $v(\cdot)$ is called a \emph{non-Archimedean} valuation.
Two valuations $v,v'$ on $K$ are \emph{equivalent} if there exists an $a\in\R_{>0}$ such that $v'(x)=v(x)^{a}$ for all $x\in K$.


Let $R_k$ denote the algebraic integers in the number field $k$. For a real embedding $\sigma:k\hookrightarrow\R$, we can define $v_\sigma(x)=|\sigma(x)|$. For a complex embedding $\sigma:k\hookrightarrow\C$, we can define $v_\sigma(x)=|\sigma(x)|^2$. For any prime ideal $\Padic$ of $R_k$, let $N(\Padic)=|R_k/\Padic|$, which is a finite number. Take $v_\Padic(x)=(1/N(\Padic))^{n_\Padic(x)}$ for $x\in R_k$, where $n_\Padic(x)$ is the largest integer $m$ such that $x\in\Padic^m$. 
If there is no such $m$, let $n_\Padic(x)=0$.
Then $v_\Padic$ can be extended to $k^*$ by letting $v_\Padic(x/y)=v_\Padic(x)/v_\Padic(y)$ for $x,y\in R_k$. 

Any Archimedean valuation on $k$ is equivalent to a valuation $v_\sigma$ for a Galois monomorphism $\sigma$ of $k$.
Any non-Archimedean valuation on $k$ is equivalent to a $\Padic$-adic valuation $v_{\Padic}$ for some prime ideal $\Padic$ in $R_k$. 

Our final goal is to work out Example \ref{example: quaternion algebra}. We start by understanding the field.

\begin{example} \label{example: prime ideals}
    Let $k=\Q(i)$. The ring of integers of $k$ is $\Z[i]$.

    We look for prime ideals of $\Z[i]$. Because a prime ideal in $\Z$ is of the form $(p)$ where $p$ is a prime, if $\Padic$ is a prime ideal in $\Z[i]$, $\Padic\cap\Z$ is also a prime ideal in $\Z$ (because if $a,b\in\Z$ and $ab\in \Padic\cap\Z$ then $a\in \Padic, b\in \Padic, a\in\Z, b\in\Z$). Therefore, $\Padic\cap \Z= (p)$ where $p$ is some prime number. So every prime ideal $\Padic\subseteq\Z[i]$ contains multiples of $p$ for some prime $p$, and hence all numbers of the form $p(a+bi)$ where $a,b\in\Z$. Note that however, $p$ might not generate the entire prime ideal $\Padic$.

    We want to investigate which prime numbers $p\in\Z$ generate a prime ideal in $\Z[i]$. To suppress notation, we still denote $(p)=\{p(a+bi):\ a,b\in\Z\}\subseteq \Z[i]$ for $p$ prime.
    Recall that $(p)$ is a prime ideal when $\Z[i]/(p)\cong(\Z/p\Z)[i]\cong (\Z/p\Z)[x]/(x^2+1)$ is an integral domain. This happens exactly when $x^2+1$ is irreducible over $\Z/p\Z$.
    
    Now, $x^2+1$ is reducible over $\Z/p\Z$ when there is an $a\in\Z_p$ such that $x^2+1=(x+a)(x-a)$, in which case there exists $a\in\Z/p\Z$ such that $a^2\equiv -1 \mod p$. Then (a consequence of) the Euler criterion (see for example \cite[Theorem 11.5]{rosen}) tells us that for an odd prime $p$, $-1$ is a quadratic residue of $p$ if and only if $p\equiv 1 \mod 4$; $-1$ is a quadratic nonresidue of $p$ if and only if $p\equiv -1\mod 4$.
    Therefore, we consider $p=2$, and odd primes $p \mod 4$: \begin{enumerate}
    \item The ideal $(2)$ is not a prime ideal in $\Z[i]$ because $2=(1+i)(1-i)$, but $(1+i)$ is a prime ideal (one can check that $\Z[i]/(1+i)\cong \Z/2\Z$).
    \item The ideal $(p)$ where $p\equiv 3\mod 4$ is a prime ideal.
    \item The ideal $(p)$ where $p\equiv 1\mod 4$ is not a prime ideal.
\end{enumerate}
\end{example}

\begin{example}
    To get an idea of what the valuations mean, we can compute a few examples. For example, $N((1+i))=2$ and $\Z[i]/(3)\cong\Z_3[i]$ so $N((3))=9$. 
    We have $(1+i)(1+i)=2i\in (1+i)^2\subseteq (1+i)$ but $2i\notin (3)$. So $n_{(1+i)}(2i)=2$, $n_{(3)}(2i)=0$, and $v_{(1+i)}(1+i)=\frac{1}{2}$, $v_{(1+i)}(2i)=(\frac{1}{2})^2$, $v_{(1+i)}(\frac{1+i}{2i})=2$, $v_{(3)}(2i)=1$.
\end{example}

For $x,y\in K$ and $v$ a valuation on $K$, defining $d(x,y)=v(x-y)$ makes $K$ a metric space. We can take the metric completion of $K$ with respect to $v$, and denote it by $K_v$.

If $v$ is Archimedean, then $K_v\cong\R$ or $\C$. If $v$ is non-Archimedean on a number field $k$, then $v$ corresponds to a prime ideal $\Padic$.
We also denote the completion of $k$ at $v_\Padic$ by $k_\Padic$, usually referred to as a \emph{$\Padic$-adic field}. 
To save notation, we still denote the extended valuation to the metric completion by $v$, $v_\sigma$, or $v_\Padic$. 

Suppose $v$ is a non-Archimedean valuation on $K$. Let $R(v)=\{\alpha\in K:\ v(\alpha)\leq 1\}$ and $\Padic(v)=\{\alpha\in K:\ v(\alpha)< 1\}$.
Then the \emph{valuation ring} $R(v)$ is a local ring whose unique maximal ideal is $\Padic(v)$ and field of fractions is $K$.

The valuation ring of $k_\Padic$ (note that this is after the metric completion) is denoted by $R_\Padic$ and called the \emph{$\Padic$-adic integers}.

The unique maximal ideal $P(v_\Padic)$ of the valuation ring $R(v_\Padic)$ of $k$ with respect to $v_\Padic$ is generated by a single element $\pi\in R_k$. Likewise, the unique maximal ideal of $R_\Padic$ is generated by $i_\Padic(\pi)$, where $i_\Padic: k\hookrightarrow k_\Padic$ is the inclusion. There is an isomorphism between the residue fields \begin{equation} \label{equation: residue field}
    R_\Padic/i_\Padic(\pi)R_\Padic\cong R(v_\Padic)/\pi R(v_\Padic).\end{equation} 
This will later allow us to do our computations inside $R(v_\Padic)/\pi R(v_\Padic)$ without having to understand what exactly $k_\Padic$ or $R_\Padic$ is.

\begin{example}
    Let $k=\Q(i)$, $R=\Z[i]$ and $\Padic=(3)$. Then $R(v_{(3)})=\{\alpha\in\Q(i):\ v_{(3)}(\alpha)\leq 1\}=\{\frac{a}{b}:\ a,b\in\Z[i], n_{(3)}(a)\geq n_{(3)}(b)\}= \{\frac{a}{b}:\ a,b\in\Z[i], b\notin (3)\}$ and $\Padic(v_{(3)})=\{\frac{a}{b}:\ a,b\in\Z[i], a\in (3), b\notin (3)\}=3R(v_{(3)})$. So $\pi=3$. 
\end{example}

We would call a real embedding or a pair of complex embeddings of $k$ \emph{a real or a complex place}. We also call an equivalence class of valuations (depending on the embedding) a place. We call the classes of Archimedean valuations \emph{infinite places}, and the classes of non-Archimedean valuations \emph{finite places}.

\subsubsection{Quaternion algebras} \label{subsection: quaternion alg}

A \emph{quaternion algebra} $\QA=\left(\frac{a,b}{K}\right)$ over a field $K$ is a 4-dimensional $K$-space with basis vectors $1,I,J,IJ$, where the multiplication is defined on $\QA$ by requiring that $1$ is a multiplicative identity, and
$I^2=a,\ J^2=b,\ IJ=-JI$,
where $a,b\in K^*$. By extending the multiplication linearly, $\QA$ is an associative algebra over $K$. If $L$ is a field extension of $K$, then $\left(\frac{a,b}{K}\right)\otimes_{K,\sigma} L\cong\left(\frac{\sigma(a),\sigma(b)}{L}\right)$.

For any field $K$, $\left(\frac{1,1}{K}\right)\cong \M_2(K)$ with generators
$$I\mapsto\begin{pmatrix}
    1 & 0 \\ 0 & -1 
\end{pmatrix},\ \text{ and } J\mapsto\begin{pmatrix}
    0 & 1 \\ 1 & 0
\end{pmatrix}.$$
For any $a,b,x,y\in K^*$, $\left(\frac{a,b}{K}\right)\cong\left(\frac{ax^2,by^2}{K}\right)$. Therefore, we always have $\left(\frac{a,b}{\C}\right)\cong\left(\frac{1,1}{\C}\right)\cong \M_2(\C)$.

\begin{definition}
    For each $x\in\QA$, the \emph{(reduced) norm} of $x$ is defined by $n(x)=x\bar x$.
    The \emph{(reduced) trace} of $x$ is defined by $\text{tr}(x)=x+\bar x$.
\end{definition}

\begin{theorem}[\cite{maclachlan-reid}, Theorem 2.1.7] \label{theorem: dichotomy of quaternion algebra}
    If $\QA$ is a quaternion algebra over $K$, then $\QA$ is either a division algebra, or $\QA$ is isomorphic to $\M_2(K)$.
\end{theorem}

Note that $\left(\frac{a,b}{\R}\right)\cong \left(\frac{1,1}{\R}\right)$, $\left(\frac{1,-1}{\R}\right)$, or $\left(\frac{-1,-1}{\R}\right)$. The first two are isomorphic to  $\M_2(\R)$, and $\left(\frac{-1,-1}{\R}\right)\cong\mathcal H$ is the Hamilton's quaternion algebra, which is a division algebra, isomorphic to the $\R$-subalgebra
$$\left\{\begin{pmatrix}
    \alpha & \beta \\ -\bar\beta & \bar\alpha
\end{pmatrix}\big|\ \alpha,\beta\in\C\right\}$$
by sending \begin{equation}\label{equation: hamilton's quaternion}
    I\mapsto\begin{pmatrix}
    i & 0 \\ 0 & -i
\end{pmatrix},\ \ \text{and } J\mapsto\begin{pmatrix}
    0 & 1 \\ -1 & 0
\end{pmatrix}
\end{equation} where we always denote by $i=\sqrt{-1}$ the imaginary unit of $\C$. The norm one elements $\mathcal H^1=\{x\in\mathcal H\ |\ n(x)=1\}$ form a group isomorphic to $\SU(2)$.

\begin{definition}
    If $\QA$ is a quaternion algebra over the number field $k$, let $\QA_v$ (resp. $\QA_\Padic$) denote the quaternion algebra $\QA\otimes_k k_v$ (resp. at $\Padic$) over $k_v$ (resp. $k_\Padic$). Then $\QA$ is said to be \emph{ramified} at $v$ (resp. $\Padic$) if $\QA_v$ (resp. $\QA_\Padic$) is the unique (see \cite[Theorems 2.5.1, 2.6.3]{maclachlan-reid}) division algebra over $k_v$ (resp. $k_\Padic$). Otherwise, we say that $\QA$ \emph{splits} at $v$ or $\Padic$.
\end{definition}

The above isomorphism is actually unique up to conjugacy by the Skolem Noether Theorem \cite[Theorem 2.9.8]{maclachlan-reid}. For a complex place $(\sigma,\bar\sigma):k\hookrightarrow\C$, we can choose the isomorphism $\QA\otimes_{k,\sigma}\C\to \M_2(\C)$ to be  
\begin{equation} \label{equation: splitting case}
    I\mapsto\begin{pmatrix}
    \sqrt{\sigma(a)} & 0 \\ 0 & -\sqrt{\sigma(a)} 
\end{pmatrix},\ \text{ and } J\mapsto\begin{pmatrix}
    0 & \sigma(b) \\ 1 & 0
\end{pmatrix}.
\end{equation}
For a real place $\sigma:k\hookrightarrow\R$, if $a>0,b>0$ we use the same map as above; if $a<0,b>0$ we simply swap the roles of $a$ and $b$ (to get a matrix in $\M_2(\R)$).

The following theorem is useful for our computation.
\begin{theorem}[\cite{maclachlan-reid}, Theorem 2.6.6]\label{theorem: criterion for P-adics} Let $K$ be a non-dyadic $\Padic$-adic field, with integers $R$ and maximal ideal $\Padic$. Let $\QA=\left(\frac{a,b}{K}\right)$, where $a,b\in R$.\begin{enumerate}
    \item If $a,b\notin\Padic$, then $\QA$ splits.
    \item If $a\notin\Padic$, $b\in\Padic\backslash\Padic^2$, then $\QA$ splits if and only if $a$ is a square mod $\Padic$.
    \item If $a,b\in\Padic\backslash \Padic^2$, then $\QA$ splits if and only if $-a^{-1}b$ is a square mod $\Padic$.
\end{enumerate}
\end{theorem}

The following criterion can be used to detect if a quaternion algebra splits over a number field.
\begin{theorem}[\cite{maclachlan-reid}, Theorem 2.7.2] \label{theorem: criterion for splitting over number field k}
    Let $\QA$ be a quaternion algebra over a number field $k$. Then $\QA$ splits over $k$ if and only if $A\otimes_k k_v$ splits over $k_v$ for all places $v$.
\end{theorem}

We have the following classification theorem of quaternion algebras.

\begin{theorem}[\cite{maclachlan-reid}, Theorem 7.3.6]\label{theorem: classification of quaternion algebras}
    Let $\QA$ be a quaternion algebra over the number field $k$ and let $\text{Ram}(\QA)$ denote the set of places (the set of infinite places $\Ram_\infty(\QA)$, the set of finite places $\Ram_f(\QA)$) at which $\QA$ is ramified. Then the following hold: \begin{enumerate}
        \item $\text{Ram}(\QA)$ is finite of even cardinality.
        \item Let $\QA_1,\QA_2$ be quaternion algebras over $k$. Then $\QA_1\cong \QA_2$ if and only if $\text{Ram}(\QA_1)=\text{Ram}(\QA_2)$.
        \item Let $S$ be any finite set of places of $\Omega(k)\backslash\{\text{non-real places in }\Omega_\infty\}$ of even cardinality. Then there exists a quaternion algebra $\QA$ over $k$ such that $\text{Ram}(\QA)=S$. Here $\Omega(k)$ denotes the set of all places of $k$.
    \end{enumerate}
\end{theorem}

\bigskip

\section{Tomter's construction} \label{section: tomter's construction}

We would like to attribute the idea of the construction of the representations largely to Tomter \cite{tomter}. The original work gives an algebraic Anosov flow that is fibrewise Anosov, over a geodesic flow in dimension 3. We believe most of the facts about arithmetic Kleinian or Fuchsian groups are well understood by experts, but we still write them in detail in order to show that they serve our purposes in constructing Anosov flows.

Throughout this section, our aim is to prove the following claim.

\begin{claim} \label{claim: existence of representation}
    There exists an arithmetic Kleinian group derived from a quaternion algebra of the form $\left(\frac{a,b}{\Q(\sqrt{-d})}\right)$ where $a,b\in\Q(\sqrt{-d})^*$ and $d\in\Z_{>0}$ is squarefree such that it acts on $\R^8$ preserving a lattice isomorphic to $\Z^8$, and is cocompact in $\SL(2,\C)$.
    It has a finite index subgroup $\Gamma$ that is torsion free, which is the fundamental group of a closed hyperbolic 3-manifold. The action of this subgroup on the 8-dimensional torus is totally Anosov. Moreover, for any $\gamma\in\Gamma$, the eigenvalues of $\rho(\gamma)$ are $e^{(\ell(\gamma)\pm i\theta(\gamma))/2}, e^{-(\ell(\gamma)\pm i\theta(\gamma))/2}$, where $\ell(\gamma)$ denotes the translation length of $\gamma$ and $\theta(\gamma)$ denotes the rotation around the axis of $\gamma$ in $\hyp^3$.
\end{claim}

We define arithmetic Kleinian groups (Definition \ref{def: arith Kleinian group}) and construct their actions in Section \ref{subsection: action on tori}.

\subsection{Arithmetic Kleinian groups} \label{subsection: arith kleinian groups}

We use the same notation as we used in Section \ref{subsection: fields and quaternion algebras}. 
Let $R$ be a Dedekind domain whose field of fractions $K$ is either a number field or a $\Padic$-adic field. 

If $V$ is a vector space over $K$, an \emph{$R$-lattice} $L$ in $V$ is a finitely generated $R$-module contained in $V$. Furthermore, $L$ is a \emph{complete} $R$-lattice if $L\otimes_R K\cong V$.

Let $\QA$ be a quaternion algebra over $K$. An \emph{order} $\calO$ in $\QA$ is a complete $R$-lattice which is also a ring with $1$, and it is also the \emph{ring of integers} in $\QA$ that contains $R$ such that $K\calO=\QA$.

\begin{theorem}[\cite{maclachlan-reid}, Theorem 8.1.1]\label{theorem: embedding into R-vector space}
    If $\QA$ is ramified at $s_1$ real places, then    $$\QA\otimes_\Q\R\cong \QA\otimes_k(k\otimes_\Q\R)\cong\oplus_{s_1}\mathcal H\oplus_{(r_1-s_1)}\M_2(\R)\oplus _{r_2}\M_2(\C).$$
\end{theorem}

Let $G=\oplus_{v\in\Omega_\infty\backslash\Ram_\infty(\QA)}\QA_v\cong\oplus \M_2(k_v)$. Let $\iota:\QA\to G$ be such that $\iota=p\circ i$ where $i:\QA\hookrightarrow\QA\otimes_\Q\R$ for any $a\in\QA$, $i(a)=a\otimes 1$ and $p:\QA\otimes_\Q\R\to G$ is the projection. $\iota$ is an embedding.

\begin{theorem}[\cite{maclachlan-reid}, Theorem 8.1.2] \label{theorem: order}
    Let $\calO$ be an order in a quaternion algebra $\QA$ that splits at one or more infinite places of $k$, and let $\calO^1:=\{\alpha\in\calO:\ n(\alpha)=1\}$. Under the embedding $\iota$ described above, $\iota(\calO^1)$ is discrete and of finite covolume in $G^1=\sum \SL(2,k_v)$. Furthermore, if $\QA$ is a quaternion division algebra, then $\iota(\calO^1)$ is cocompact. Also, if $G'=\sum \SL(2,k_v)$ is a factor of $G^1$ with $1\neq G'\neq G^1$, then the projection of $\iota(\calO^1)$ in $G'$ is dense in $G'$.
\end{theorem}
This implies that if there exists $G'$ such that $1\neq G'\neq G^1$, then $\iota(\calO^1)$ is dense in $G'$. Therefore, if we want a lattice in $\SL(2,\R)$ or $\SL(2,\C)$ that acts on $\hyp^2$ or $\hyp^3$, respectively, $\QA$ must have only one infinite place at which it splits. $G^1$ is a locally compact group.

\begin{definition}[\cite{maclachlan-reid}, Definition 8.2.1]\label{def: arith Kleinian group}
    Let $k$ be a number field with exactly one complex place and let $\QA$ be a quaternion algebra over $k$ which is ramified at all real places. Let $\iota$ be a $k$-embedding of $\QA$ into $\M_2(\C)$ and let $\calO$ be an $(R_k-)$order of $\QA$. Then a subgroup $\Gamma$ of $\SL(2,\C)$ (or $\PSL(2,\C)$) is an \emph{arithmetic Kleinian group} if it is commensurable with some such $\iota(\calO^1)$ or $P\iota(\calO^1)$. Hyperbolic $3$-manifolds and $3$-orbifolds, $\hyp/\Gamma$, will be referred to as \emph{arithmetic} when their covering groups $\Gamma$ are arithmetic Kleinian groups.
\end{definition}
Similarly, let $\QA$ be a quaternion algebra over a totally real field $k$ that is ramified at all real places except for one. Let $\iota$ be a $k$-embedding of $\QA$ into $\M_2(\R)$ and $\calO$ an order in $\QA$. Then a subgroup of $\SL(2,\R)$ is an \emph{arithmetic Fuchsian group} if it is commensurable with some $\iota(\calO^1)$ or $P\iota(\calO^1)$.

From Theorem \ref{theorem: order}, if we want to look for arithmetic Kleinian or Fuchsian groups, the quaternion algebra which we start with must be such that $\QA\otimes_\Q\R\cong \oplus_{s_1}\mathcal H\oplus \M_2(k_v)$ where $k_v$ is either $\C$ or $\R$, respectively. Meanwhile, since we want a cocompact lattice, we need $\QA$ to be a division algebra. From Theorems \ref{theorem: dichotomy of quaternion algebra} and \ref{theorem: criterion for splitting over number field k}, we need $\QA$ to be ramified at some even number of places of $k$. We will show that these places of $k$ where $\QA$ is ramified have to be finite places in order for us to construct fibrewise Anosov flows.

\subsection{Arithmetic Kleinian groups acting on tori} \label{subsection: action on tori}

Now suppose that we have a quaternion algebra $\QA\hookrightarrow\QA\otimes_\Q\R\cong \oplus_{s_1}\mathcal H\oplus \M_2(k_v)$, where $k_v$ is either $\C$ or $\R$, which gives rise to the arithmetic Kleinian or Fuchsian group we want. Then the groups given by the orders act on a higher dimensional real vector space preserving an integer lattice. The idea is ``restriction of scalars''.
We have omitted the proof of Theorem \ref{theorem: embedding into R-vector space}, but we describe the isomorphism here which is useful for us.

When $k$ is totally real, $\QA\otimes_\Q\R\cong \oplus_{s_1}\mathcal H\oplus \M_2(\R)$.
Let $\sigma_l,\ l=1,\ldots, s_1,s_1+1$ denote the embeddings $k\hookrightarrow\R$, where the first $s_1$ are ramified places and $\sigma_{s_1+1}$ is the split real place. Let $\QA_l=\left(\frac{\sigma_l(a),\sigma_l(b)}{\R}\right)$, $l=1,\ldots,s_1+1$ and $1_l,I_l,J_l,I_lJ_l$ be its basis. Define ring homomorphisms $\hat\sigma_l:\QA\to \QA_l$, such that 
$$\hat\sigma_l(x_01+x_1I+x_2J+x_3IJ)=\sigma_l(x_0)1_l+\sigma_l(x_1)I_l+\sigma_l(x_2)J_l+\sigma_l(x_3)I_lJ_l.$$ 
When $k_v=\C$, $\QA\otimes_\Q\R\cong \oplus_{s_1}\mathcal H\oplus \M_2(\C)$. We define $\hat\sigma_l,\ l=1,\ldots,s_1$ as in the above real case, but we let $\hat\sigma_{s_1+1}:\QA\to \QA_{s_i+1}\oplus \bar\QA_{s_i+1}$, where $ \QA_{s_i+1}=\left(\frac{\sigma_{s_1+1}(a),\sigma_{s_1+1}(b)}{\C}\right), \bar\QA_{s_i+1}=\left(\frac{\bar\sigma_{s_1+1}(a),\bar\sigma_{s_1+1}(b)}{\C}\right)$ be such that
\begin{align*}
    & \hat\sigma_{s_1+1}(x_01+x_1I+x_2J+x_3IJ) \\
    =\; & \sigma_{s_1+1}(x_0)1_{s_1+1}+\sigma_{s_1+1}(x_1)I_{s_1+1}+\sigma_{s_1+1}(x_2)J_{s_1+1}+\sigma_{s_1+1}(x_3)I_{s_1+1}J_{s_1+1} \\
    & \qquad +\bar\sigma_{s_1+1}(x_0)1'_{s_1+1}+\bar\sigma_{s_1+1}(x_1)I'_{s_1+1}+\bar\sigma_{s_1+1}(x_2)J'_{s_1+1}+\bar\sigma_{s_1+1}(x_3)I'_{s_1+1}J'_{s_1+1},
\end{align*}
where $1_{s_1+1},I_{s_1+1},J_{s_1+1},I_{s_1+1}J_{s_1+1}$ and $1_{s_1+1}',I_{s_1+1}',J_{s_1+1}',I_{s_1+1}'J_{s_1+1}'$ denote the bases of $\QA_{s_1+1}$ and $\bar\QA_{s_1+1}$, respectively. Note that $\hat\sigma_l, l=1,\ldots,s_1+1$ are all $\Q$-linear. 

Next, we let $\hat\phi:\QA\otimes_\Q\R\to \oplus_{s_1+1}\QA_l$ (if $k$ is totally real) or $\hat\phi:\QA\otimes_\Q\R\to\oplus_{s_1}\QA_l\oplus(\QA_{s_i+1}\oplus\bar\QA_{s_i+1})$ (if $k$ has a complex place) be such that $\hat\phi(x\otimes r)=(r\hat\sigma_1(x),\ldots,r\hat\sigma_{s_1+1}(x))$ for $x\in\QA$ and $r\in\R$.
Let the multiplication on $\QA\otimes_\Q\R$ be $(x\otimes r_1)(y\otimes r_2)=(xy)\otimes (r_1r_2)$ and the multiplication on $\oplus_{s_1+1}\QA_l$, and $\oplus_{s_1}\QA_l\oplus(\QA_{s_i+1}\oplus\bar\QA_{s_i+1})$ be coordinate-wise.
One can check that $\hat\phi$ preserves the multiplication and is $\R$-linear. Recall that in the previous section, we had either an isomorphism $\QA\otimes_{k,\sigma} k_v\cong \M_2(k_v)$ if $\sigma$ is a real or complex place at which $\QA$ splits (see Equation (\ref{equation: splitting case})), or a map that maps a Hamilton's quaternion algebra onto a subalgebra of $\M_2(\C)$ (see Equation (\ref{equation: hamilton's quaternion})). Composing these maps with $\hat\phi$, this is our isomorphism $\phi:\QA\otimes_\Q\R\to \oplus_{s_1}\mathcal H\oplus \M_2(k_v)$. In the last summand, if $k$ has a complex place, we actually get $\Delta(\M_2(\C))\cong \M_2(\C)$, where the isomorphism is given by the diagonal map $x\mapsto (x,\bar x)$ for any $x\in \M_2(\C)$. In either case, $\oplus_{s_1}\mathcal H\oplus \M_2(k_v)$ is a real vector space of dimension $4n$, where $n$ is the degree of $k$ over $\Q$.

Let us denote the projection to each summand $\phi_l:=p_l\circ\phi\circ i:\QA\to \mathcal H$ or $\M_2(k_v)$, $l=1,\ldots, s_1+1$.
If $x=x_01+x_1I+x_2J+x_3IJ\in \left(\frac{a,b}{k}\right)$, $n(x)=x\bar x=x_0^2-ax_1^2-bx_2^2+abx_3^2$.  From Equations (\ref{equation: hamilton's quaternion}) and (\ref{equation: splitting case}), and the fact that all such isomorphisms are conjugate to each other, we can check that $\sigma_l(n(x))=\det(\phi_l(x))$.
Also note that $\phi$ and $\phi_l$ are all $\R$-linear.

We want to check that $\phi\circ i(\calO)$ is a complete $\Z$-lattice in $\oplus_{s_1}\mathcal H\oplus \M_2(k_v)$ of dimension $4n$.
By definition, $\calO$ is an $R_k$-module of $\QA$ where $R_k$ is the ring of integers of $k$ and has a $\Z$-basis of dimension $n$, denoted
by $\{\alpha_1,\ldots,\alpha_n\}\subseteq R_k$. Then $\calO$ has the $\Z$-basis $\alpha_le_1, \alpha_le_2,\alpha_le_3, \alpha_le_4$ of dimension $4n$, where $e_i,\ i=1,2,3,4$ can be $k$-linear combinations of $1,I,J,IJ$.
But for $l,m=1,\ldots,n$, we have
$$\phi_l(\alpha_me_1)=\sigma_l(\alpha_m)\phi_l(e_1),\  \phi_l(\alpha_me_2)=\sigma_l(\alpha_m)\phi_l(e_2),$$
$$\phi_l(\alpha_me_3)=\sigma_l(\alpha_m)\phi_l(e_3),\  \phi_l(\alpha_me_4)=\sigma_l(\alpha_m)\phi_l(e_4),$$
and $\phi$ is a vector space isomorphism and is in particular $\Z$-linear. For instance, vectors of the form
$$\phi(\alpha_m e_i)=(\sigma_1(\alpha_m)\phi_1(e_i),\sigma_2(\alpha_m)\phi_2(e_i),\ldots,\sigma_n(\alpha_m)\phi_n(e_i)),\ i=1,2,3,4$$
indeed form a $4n$-dimensional $\Z$-basis of $\oplus_{s_1}\mathcal H\oplus \M_2(k_v)$.


Now, $\QA$ acts on $\QA\otimes_\Q\R$ by left multiplication by elements of $\QA\otimes_\Q 1$, such that $x.(y\otimes r)=(xy)\otimes r$ for any $x,y\in\QA$ and $r\in\R$. This induces the action of $\phi(\QA\otimes_\Q 1)$ on $\phi(\QA\otimes_\Q\R)$, because $\phi$ preserves the multiplication. This action restricts to $\calO^1$ and $\phi\circ i(\calO^1)$ preserves $\phi\circ i(\calO)$ simply because $\calO$ is a ring.
We can think of $\oplus_{s_1}\mathcal H\oplus \M_2(k_v)$ as consisting of block diagonal matrices, and the $\phi\circ i(\calO^1)$ action on $\oplus_{s_1}\mathcal H\oplus \M_2(k_v)$ preserves each block (as it is just matrix multiplication).

Note that for any matrix $\gamma =\begin{pmatrix}
    x & y \\ z & w
\end{pmatrix}\in \M_2(\R)$, if we consider it to act on the $4$-dimensional $\R$-vector space spanned by the basis $\begin{pmatrix}
    1 & 0 \\ 0 & 0
\end{pmatrix}, \begin{pmatrix}
    0 & 0 \\ 1 & 0
\end{pmatrix}, \begin{pmatrix}
    0 & 1 \\ 0 & 0
\end{pmatrix}, \begin{pmatrix}
    0 & 0 \\ 0 & 1
\end{pmatrix}$, the action corresponds to the linear operator
$$L_\gamma:=\begin{pmatrix}
    x & y & 0 & 0 \\
    z & w & 0 & 0 \\
    0 & 0 & x & y \\
    0 & 0 & z & w
\end{pmatrix}\in \M_4(\R),$$
so $\det(L_\gamma)=(\det(\gamma))^2$. Similarly, for $\gamma\in \M_2(\C)$, it acts as a linear operator $L_{\gamma,\C}\in \M_4(\C)$, which has determinant $\det(L_{\gamma,\C})=\det(\gamma)^2 \in\C$.
On the other hand, we can write $L_{\gamma,\C}=L_1+iL_2$ for some $L_1, L_2\in \M_4(\R)$, and let $u+iv\in\C^4\cong\R^4\oplus i\R^4$ where $u,v\in\R^4$.
Then $(L_1+ iL_2)(u+ iv)=(L_1u-L_2v)+ i(L_1v+L_2u)$. If we consider $L_{\gamma,\C}$ acting on $\R^4\oplus i\R^4$, it can also be written as a real matrix $L_{\gamma,\R}=\begin{pmatrix}
    L_1 & -L_2 \\ L_2 & L_1
\end{pmatrix}\in \M_8(\R)$ and it is conjugate to $L_\gamma:=\begin{pmatrix}
    L_1+ iL_2 & 0 \\ 0 & L_1-  i L_2
\end{pmatrix}$ in $\M_8(\C)$ with determinant $|\det(L_{\gamma,\C})|^2$.
So $\det(L_\gamma)=\det(L_{\gamma,\R})=|\det(\gamma)|^4$.

For the Hamilton's quaternion factors, $x\in\mathcal H_l, l=1,\ldots,s_1$, it is easier to compute directly in the basis $1,I,J,IJ$, and we also have $\det(L_{\phi_l(x)})=\sigma_l(n(x))^2\in\R$.
Therefore, for any $x\in\calO$, if we let $\mathcal L_x\in \M_{4n}(\R)$ denote the linear operator of $x$ acting on $\oplus_{s_1}\mathcal H\oplus \M_2(k_v)$, we have $$\det(\mathcal L_x)=\prod_{l=1}^{s_1+1} \det(L_{\phi_l(x)})= |\sigma_{s_1+1}(n(x))|^4\prod_{l=1}^{s_1}\sigma_l(n(x))^2,\ \text{for the Kleinian group case},$$
and $$\det(\mathcal L_x)=\sigma_{s_1+1}(n(x))^2\prod_{l=1}^{s_1}\sigma_l(n(x))^2,\ \text{for the Fuchsian group case}.$$
Restricting to $x\in\calO^1$, $\mathcal L_x\in \SL(4n,\Z)$.

\begin{definition}
    We denote the action of $\calO^1$ on $\T^{4n}$ by $\rho:\calO^1\to \SL(4n,\Z),\ x\mapsto \mathcal L_x$.
\end{definition}

The above computation also gives us the characteristic polynomial of $\mathcal L_x$,
\begin{align*}
    \chi_{\mathcal L_x}(\lambda) =\det(\mathcal L_x-\lambda \Id) =& \prod_{l=1}^{s_1+1}\det(L_{\phi_l(x)}-\lambda \Id) \\
    = &\det(\phi_{s_1+1}(x)-\lambda \Id)^2\det(\overline{\phi_{s_1+1}(x)}-\lambda \Id)^2\prod_{l=1}^{s_1}(\det(\phi_l(x)-\lambda \Id))^2
\end{align*}
for the Kleinian case, or
$$\chi_{\mathcal L_x}(\lambda)=(\det(\phi_{s_1+1}(x)-\lambda \Id))^2\prod_{l=1}^{s_1}(\det(\phi_l(x)-\lambda \Id))^2$$
for the Fuchsian case.

Therefore, the moduli of the eigenvalues of $\mathcal L_x$ agree with those of the eigenvalues of $x$.

\begin{lemma} \label{lemma: SU factor}
    If $\QA$ is ramified at a real place, the action has an eigenvalue on the unit circle.
\end{lemma}
\begin{proof}
    We can compute explicitly the eigenvalues for one of the Hamilton's quaternion factors. Recall the isomorphism (\ref{equation: hamilton's quaternion})
    $$\mathcal H^1\cong \left\{\begin{pmatrix}
        \alpha & \beta \\ -\bar\beta & \bar\alpha
    \end{pmatrix}\big|\ \alpha,\beta\in\C,\ |\alpha|^2+|\beta|^2=1 \right\}.$$
    Then for $x\in\mathcal H$,
    $\chi_x(\lambda)=\lambda^2-(\alpha+\bar\alpha)\lambda+1$.
    Solving for $\lambda$, we have $\lambda=\frac{2\text{Re}(\alpha)\pm\sqrt{4\text{Re}(\alpha)^2-4}}{2}$.
    Since $|\text{Re}(\alpha)|\leq |\alpha|\leq 1$, $\sqrt{4\text{Re}(\alpha)^2-4}$ is purely imaginary or $0$. In the latter case, $x=\pm \Id$.
    We compute for the first case,
    $$\lambda\bar\lambda=\frac{4\text{Re}(\alpha)^2-4\text{Re}(\alpha)^2+4}{4}=1.$$
    Therefore $|\lambda|=1$. In either case, $x$ has an eigenvalue on the unit circle.
\end{proof}

\begin{proof}[Proof of Claim \ref{claim: existence of representation}]
    The quaternion algebra that we want must have no real places by Lemma \ref{lemma: SU factor} and exactly one complex place, and to generate a cocompact arithmetic Kleinian group, it must be ramified at an even number of finite places by Theorem \ref{theorem: classification of quaternion algebras} and Theorem \ref{theorem: order}. 
    Therefore the field must be of the form $\Q(\sqrt{-d})$ where $d>0$ is a real number.
    Such examples exist also by the Classification Theorem \ref{theorem: classification of quaternion algebras}. Namely, take the ring of integers of $\Q(\sqrt{-d})$. It has infinitely many nonzero prime ideals. For any even number of prime ideals, there exists a quaternion algebra that is ramified at these finite places, which gives rise to an arithmetic Kleinian group that we want. 
    By Selberg's lemma \cite[Theorem 1.3.5]{maclachlan-reid}, any arithmetic Kleinian group has a finite index subgroup that is torsion free. 
    Then this subgroup is the fundamental group of a closed hyperbolic 3-manifold, and hence every nontrivial element as an element of $\SL(2,\C)$ is loxodromic (see, for example, \cite{parker}). Therefore the action on the torus has no eigenvalues on the unit circle either.

    For the uniform hyperbolicity with respect to the translation length, we take the standard axis $(0,\infty)$ in $\hyp^3$. Without loss of generality, we can assume that it is also the axis $A_{\gamma_0}$ for some class $\gamma_0\in \pi_1M\cong \iota(\calO^1)<\SL(2,\C)$. 
    Then $\gamma_0$ is conjugate to $\begin{pmatrix}
        e^{(\ell(\gamma_0)+i\theta(\gamma_0))/2} & \\ & e^{-(\ell(\gamma_0)+i\theta(\gamma_0))/2}
    \end{pmatrix}$, as the unit speed geodesic flow acts on this axis by the M\"{o}bius transformation of $\begin{pmatrix}
        e^{t/2} & \\ & e^{-t/2}
    \end{pmatrix}$.    
    For any $\gamma\in \iota(\calO^1)$, there exists an isometry $h\in \PSL(2,\C)$ such that $hA_{\gamma_0}=A_{\gamma}$ and $A_{h^\inv\gamma h} = A_{\gamma_0}$. Thus, $h^\inv\gamma h$ also corresponds to a translation of distance $\ell(\gamma)$ along $A_{\gamma_0}$ and some rotation around $A_{\gamma_0}$. Because of the definition of $\rho$, the complex conjugates are also the eigenvalues of $\rho(\gamma)$, and each of the eigenvalues is repeated twice.
\end{proof}

We would like to emphasize that the existence and abundance of such examples are guaranteed by the Classification Theorem \ref{theorem: classification of quaternion algebras}. But explicit examples should not be hard to find and here we can give one.

\begin{example} \label{example: quaternion algebra}
    Let $\QA= \left(\frac{i+1,3}{\Q(i)}\right)$.

    The minimal polynomial of $i$ is $x^2+1$. It has no real roots and a pair of complex roots.

    We want to check that $\text{Ram}(\QA)\neq \emptyset$, that is, there exists a prime ideal $\Padic$ at which $\QA$ does not split.
    Recall our discussion of prime ideals of $\Q(i)$ in Example \ref{example: prime ideals}.
    Consider $\Padic=(3)$ and let $a=1+i$, $b=3$.
    Then we can use Theorem \ref{theorem: criterion for P-adics} (2), because we have $1+i\notin (3)$, $3\notin \Padic^2=(9)$ and the isomorphism (\ref{equation: residue field}). We show that $1+i$ is not a square mod $(3)$. But explicitly, $$\Z[i]/(3)=\{a+bi:\ a,b\in\Z_3\}=\{0,1,2,i,2i,1+i,1+2i,2+i,2+2i\}.$$
    We can check every square: $2^2= 1$, $i^2=-1=2$, $(2i)^2=-4=2$, $(1+i)^2=2i$, $(1+2i)^2=i$, $(2+i)^2=i$, $(2+2i)^2=2i$. Therefore, $1+i$ is not a square and $\QA$ does not split. 
\end{example}

\begin{remark}
    If we do not mind reducible representations, we can take representations of dimension $4n$ in the Fuchsian case and $8n$ in the Kleinian case.
    
    Tomter's \cite{tomter} original work gives algebraic examples of fibrewise Anosov flows over geodesic flows on the unit tangent bundles of hyperbolic surfaces arising in this way, and the total space has dimension $3+4n$. Now we can see that a similar construction also gives us algebraic examples of fibrewise Anosov flows in dimensions $5+8n$.
\end{remark}

\bigskip

\section{Fibrewise Anosov flows over \cite{fenley-mann-potrie}} \label{section: proof of the main theorem}

\subsection{Proof of Theorem \ref{theorem: main theorem}}

We first fix some notation. Suppose $M$ is a closed hyperbolic 3-manifold and $\widetilde M$ is its universal cover. Let $\Gamma$ denote its fundamental group that satisfies (1) of the assumptions in Theorem \ref{theorem: main theorem}. 
We can construct fibrewise flows using the representation obtained in Section \ref{section: tomter's construction}. Let us denote this representation by $\rho:\Gamma\to \SL(d,\Z)$, where $d=8n$ for some positive integer $n$ (more precisely, $\rho$ takes values in $\oplus_n \SL(8,\Z)$ when it is reducible).
Let $\varphi^t:M'\to M'$ be the Anosov flow on a flat $F$-bundle $M'$ over $M$, which lifts equivariantly to $\tilde\varphi^t:\widetilde M\times F\to \widetilde M\times F$. Let $\tilde\Phi^t: (\widetilde M\times F)\times\T^d\to (\widetilde M\times F)\times\T^d$ be such that $$\tilde\Phi^t(x,y)=(\tilde\varphi^t(x),y),$$
and $N:=\Gamma\backslash((\widetilde M\times F)\times\T^d)$ where $\gamma\in\Gamma$ acts by $\gamma.(x,y)=(\gamma.x,\rho(\gamma) y)$.
The flow $\tilde\Phi^t$ is automatically equivariant, since it is the identity in the fibres. It projects to a flow $\Phi^t: N\to N$. 
We denote the projection as $\pi_0:\widetilde M\times F\to \widetilde M$.

\begin{lemma}
    Assume that a lifted curve $\ell$ from $M'$ is a $k$ quasi-geodesic in $\widetilde M\times F$ with respect to the metric on $\widetilde M$. Then $\pi_0(\ell)$ is a $k'$ quasi-geodesic in $\widetilde M$ for some $k'\geq k$.
\end{lemma}
\begin{proof}
    With some abuse of notation, we do not distinguish between a Riemannian metric $g$ and its induced norm $\|\cdot\|_g$ on a tangent space.
    Because $M'$ is compact, every Riemannian metric $\|\cdot\|_g$ on $M'$ is equivalent to the product metric, defined by $(\|v\|')^2=\|d\pi_0v\|_M^2+\|d\pi_Fv\|_F^2$ for a vector $v\in TM'$, where $\pi_F$ denotes the projection to the fibres $F$ and $\|\cdot\|_M$ and $\|\cdot\|_F$ are fixed Riemannian metrics of $M$ and $F$ (under the metric $\|\cdot\|'$ the tangent space of $M$ and that of $F$ are orthogonal). This means that there exists a constant $c\geq 1$ such that
    $$\frac{1}{c}\|\cdot\|_g\leq \|\cdot\|'\leq c\|\cdot\|_g.$$
    The metric lifted to $\widetilde M\times F$ equivariantly is also equivalent to a product metric. Let us denote the product metric on $\widetilde M\times F$ by $\|\cdot\|^2=\|\cdot\|_{\widetilde M}^2+\|\cdot\|_F^2$ ($\|\cdot\|_F$ can vary in different fibres, of course) and the lifted metric by $\|\cdot\|_{\tilde g}$.

    For any curve $\ell$ lifted from $M'$, let $p,q\in\ell$ be two points. We want to show that there exists a constant $k'$, such that
    $$\frac{1}{k'}d_{\pi_0(\ell)}(\pi_0(p),\pi_0(q))-k'\leq d_{\widetilde M}(\pi_0(p),\pi_0(q)) \leq k'd_{\pi_0(\ell)}(\pi_0(p),\pi_0(q))+k',$$
    where $d_{\pi_0(\ell)}$ denotes the distance along $\pi_0(\ell)$, with respect to $\|\cdot\|_{\tilde g}$.
    We denote the distance with respect to $\|\cdot\|$ along $\ell$ and $\pi_0(\ell)$ by $d'_{\ell}$ and $d'_{\pi_0(\ell)}$ respectively.
    Because the distance along $\ell$ or $\pi_0(\ell)$ is simply integrating along $\ell$ or $\pi_0(\ell)$ with respect to the corresponding Riemannian metrics,
    we have 
    $$\frac{1}{c}d_{\widetilde M}(\pi_0(p),\pi_0(q))\leq \frac{1}{c}d_{\pi_0(\ell)}(\pi_0(p),\pi_0(q)) \leq d'_{\pi_0(\ell)}(\pi_0(p),\pi_0(q))\leq d'_{\ell}(p,q)\leq cd_\ell(p,q).$$

    Now, $\ell$ is a quasi-geodesic with respect to the metric on $\widetilde M$. By our definition, there exists a $k$ such that 
    $$\frac{1}{k}d_{\ell}(p,q)-k\leq d_{\widetilde M}(\pi_0(p),\pi_0(q)) \leq kd_{\ell}(p,q)+k.$$
    Then we have
    $$\frac{1}{c^2k}d_{\pi_0(\ell)}(\pi_0(p),\pi_0(q))-k\leq \frac{1}{k}d_\ell(p,q)-k\leq d_{\widetilde M}(\pi_0(p),\pi_0(q))\leq d_{\pi_0(\ell)}(\pi_0(p),\pi_0(q)).$$
    We can take $k'=c^2k\geq k$, and $\pi_0(\ell)$ is $k'$-quasi-geodesic in $\widetilde M$.
\end{proof}

Note that in the above proof, $c,k'$ depend only on the Riemannian metric. Therefore, if a flow on $M'$ is uniformly quasi-geodesic with respect to the metric on $\widetilde M$, then its orbits are uniformly quasi-geodesic when projected to $M$.

\subsubsection{More linear algebra} \label{subsubsection: more linear alg}

Recall that our representation $\rho$ by construction is the restriction of 
$$\hat\rho:\SL(2,\C)\to \oplus_n \SL(4,\C)\hookrightarrow \oplus_n \SL(8,\R)$$
where for $g\in \SL(2,\C)$ and $A=\oplus_n A_n\in \oplus_n \M_2(\C)$, $\hat\rho(g)\cdot A=\oplus_n gA_n$. Here $\oplus_n A_n$ denotes the diagonal block matrices with 2-by-2 blocks $A_n\in \M_2(\C)$. For simplicity of writing, we consider the case where $n=1$ below. For reducible representations with $n>1$, it works the same.

There is the natural inner product on $\M_2(\C)$ as an 8-dimensional vector space over $\R$, that is, for any $A,B\in \M_2(\C)$, we let
\begin{equation} \label{equation: matrix inner product}
    \langle A, B\rangle= \mathrm{Re\,tr}(A^*B).
\end{equation}
We can check that 
\begin{align*}
    A^*B=\begin{pmatrix}
        a_1 & a_3 \\ a_2 & a_4
    \end{pmatrix}^* \begin{pmatrix}
        b_1 & b_3 \\ b_2 & b_4
    \end{pmatrix}
    &= \begin{pmatrix}
        \bar a_1 & \bar a_2 \\ \bar a_3 & \bar a_4
    \end{pmatrix}\begin{pmatrix}
        b_1 & b_3 \\ b_2 & b_4
    \end{pmatrix} 
    = \begin{pmatrix}
        \bar a_1 b_1 + \bar a_2 b_2 & \bar a_1 b_3 + \bar a_2 b_4 \\
        \bar a_3 b_1 + \bar a_4 b_2 & \bar a_3 b_3 + \bar a_4 b_4
    \end{pmatrix},
\end{align*}
and $\mathrm{Re\, tr}(A^*B)=\mathrm{Re}(\bar a_1b_1+\bar a_2b_2+\bar a_3 b_3+\bar a_4b_4)$, which is simply the real part of the Hermitian inner product of two vectors $A,B\in\C^4$, and also the standard inner product of $A,B$ considered as vectors in $\R^8$, if we write out the real and imaginary components of $A,B$. Note that this inner product is preserved by $\SU(2)$. 

Recall that we have the identification $\hyp^3\cong \SL(2,\C)/\SU(2)$ and $\partial\hyp^3\cong\C P^1$. 
Again, consider the M\"{o}bius transformation that gives the unit speed translation on the standard axis $(0,\infty)$ in the upper half plane model $g_t=\begin{pmatrix}
    e^{t/2} & \\
     & e^{-t/2}
\end{pmatrix}$.
We note that
$\begin{pmatrix}
    e^{t/2} & \\
     & e^{-t/2}
\end{pmatrix}\begin{pmatrix}
    a_1 & a_3 \\ a_2 & a_4
\end{pmatrix}= \begin{pmatrix}
    e^{t/2}a_1 & e^{t/2} a_3 \\ e^{-t/2}a_2 & e^{-t/2} a_4
\end{pmatrix}$, which suggests the stable and unstable vector subspaces.

Let $$E^+:=\left\{\begin{pmatrix}
    a_1 & a_3 \\  0 & 0
\end{pmatrix}\,| \, a_1,a_3\in \C\right\}, \quad \ E^-:=\left\{\begin{pmatrix}
    0 & 0 \\  a_2 & a_4
\end{pmatrix}\,| \, a_2,a_4\in \C\right\} \quad \text{and we see that}$$
\begin{align*}
    & \left\{\begin{pmatrix}
    p_1 & p_3 \\  0 & p_4
\end{pmatrix}\in \SL(2,\C)\right\}=\left\{\begin{pmatrix} p_1 & p_3  & & \\ & p_4 & & \\ & & p_1 & p_3 \\ & & & p_4
\end{pmatrix}\in \SL(4,\C)\right\} \\
\subset & \, P^+:=\left\{p=\begin{pmatrix}
    A & B \\ 0 & D
\end{pmatrix}\,|\, A,B,D\in \M_2(\C),\ \det(p)=1\right\}
\end{align*}
preserves $E^+$, and
\begin{align*}
    & \left\{\begin{pmatrix}
    p_1 & 0 \\  p_2 & p_4
\end{pmatrix}\in \SL(2,\C)\right\}=\left\{\begin{pmatrix} p_1 & & & \\ p_2 & p_4 & & \\ & & p_1 & \\ & & p_2 & p_4
\end{pmatrix}\in \SL(4,\C)\right\} \\
\subset & \, P^-:=\left\{p=\begin{pmatrix}
    A & 0 \\ C & D
\end{pmatrix}\,|\, A,C,D\in \M_2(\C),\ \det(p)=1\right\}
\end{align*}
preserves $E^-$.

Taking the point $x_o=\begin{pmatrix}
    1 \\ 0
\end{pmatrix}\in\partial\hyp^3$, we can consider it as a complex line $
\C x_o\in\C^2$. Now we let
$$\xi^+(x_o):=\{A\in \M_2(\C)\,|\, \mathrm{im}(A)\subset \C x_o\}=\{\begin{pmatrix}
    a_1x_o & a_3x_o
\end{pmatrix}\in \M_2(\C)\, |\, a_1,a_3\in\C \}=E^+.$$
For any $x\in\partial\hyp^3$, there exists $g\in \SL(2,\C)$ such that $x=gx_o$, and we can define
\begin{align*}
    \xi^+(x)&=\xi^+(g.x_o) = \{\begin{pmatrix}
    a_1 g x_o & a_3 g x_o
    \end{pmatrix}\in \M_2(\C)\, |\, a_1,a_3\in\C\} \\
    &= \{g\begin{pmatrix}
    a_1 x_o & a_3 x_o
    \end{pmatrix}\in \M_2(\C)\, |\, a_1,a_3\in\C\}=g.\xi^+(x_o).
\end{align*}
This defines a map $\xi^+: \partial\hyp^3\to \mathrm{Gr_2(\C^4)}\cong \SL(4,\C)/P^+$. The isomorphism $\SL(4,\C)/P^+\to \mathrm{Gr}_2(\C^4)$ can be given by $gP^+\mapsto gE^+$. Similarly, we can define $\xi^-:\partial\hyp^3\to \mathrm{Gr}_2(\C^4)\cong \SL(4,\C)/P^-$ such that we let $y_o=\begin{pmatrix}
    0 \\ 1
\end{pmatrix}$, and at every $y=g.y_o$, $\xi^-(y)=\{g\begin{pmatrix}
    a_2 y_o & a_4 y_o
\end{pmatrix}\in \M_2(\C)\,| \, a_2, a_4\in \C\}$.
Note that $y_o=Jx_o$  where $J=\begin{pmatrix}
     & -1 \\ 1 &
\end{pmatrix}\in \SL(2,\C)$. Therefore, $\xi^-(y_o)=J.\xi^+(x_o)$ by definition.

We have for any $x\neq y$, we can write $x=g_1x_o$, $y=g_2 y_o=g_2 J x_o$ and $(g_2J)^{-1}g_1\notin P^+$. Then $\xi^+(x)=g_1.\xi^+(x_o)$ and $\xi^-(y)=g_2.\xi^-(y_o)=g_2J.\xi^+(x_o)$, so $\xi^+(x)\oplus \xi^-(y)=\C^4$. 

Now we want to understand the inclusion $\SL(4,\C)/P^\pm\hookrightarrow \SL(8,\R)/P^\pm_\R$ where $P^\pm_\R$ is the stabilizer in $\SL(8,\R)$ of $E^\pm$ considered as a real 4-dimensional plane. Recall that we have $\iota :\SL(4,\C)\hookrightarrow \SL(8,\R)$, an isomorphism onto its image, where for $L= L_1+iL_2\in \SL(4,\C)$ and $L_1,L_2\in \M_4(\R)$ we have $\iota (L)= \begin{pmatrix}
    L_1 & -L_2 \\ L_2 & L_1
\end{pmatrix} \in \SL(8,\R)$. Then this induces the inclusion $\iota(gP^\pm)=\iota(g)P^\pm_\R$, which is well-defined because $\iota(P^\pm)\subset P^\pm_\R$.

We thus have the \emph{boundary maps} to the real Grassmannians, which we still denote by 
\begin{equation} \label{equation: boundary maps}
    \xi^\pm: \partial\hyp^3\to \mathrm{Gr}_4(\R^8)\cong \SL(8,\R)/P_\R^\pm.
\end{equation}

We can also see that $\xi^\pm$ are continuous. Suppose $x_n=g_nx_o\to x=gx_o$. Then by definition $\xi^+(x_n)\to\xi^+(x)$. Similarly for $\xi^-$.

\subsubsection{The metric on $N$} \label{subsubsection: metric on N}

We want to first define a metric on $(\widetilde M\times F)\times\T^d$ that is equivariant under the $\Gamma$ action, so that it projects to $N$.

First, fix a Riemannian metric on $M'$ and lift it to $\widetilde M\times F$. Denote it by $\|\cdot\|_h$ and its induced metric $d_h$. 
We define a metric $\|\cdot\|$ on $(\widetilde M\times F)\times \T^d$ to be such that for any vector $w\in T((\widetilde M\times F)\times \T^d)$, we write $w=w_h+w_v$ where $w_h\in T(\widetilde M\times F)$ and $w_v\in T\T^d$, and let $\langle w_h,w_v\rangle =0$ and $\|w\|^2=\|w_h\|_h^2+\|w_v\|_v^2$. Here, $\|\cdot\|_v$ is defined in the following way. We also write $\|w\|_h=\|w_h\|=\|w_h\|_h$ and $\|w\|_v=\|w_v\|=\|w_v\|_v$.

Up to scaling of the metric on $\widetilde M$, under which $\varphi^t$ still stays Anosov, consider the identification $\widetilde M\cong \SL(2,\C)/\SU(2)$.
We pick the standard basepoint $o\in \widetilde M$ that is stabilized by $\SU(2)$ and, over $o$, we pick a metric $\|\cdot\|_o$ on $\T^d$ induced from the Euclidean metric. 

For any point $x\in \widetilde M$, there exists some $g\in \SL(2,\C)$ such that $go=x$.
Over $x$, we define for any $w\in T\T^d_x$,
\begin{equation} \label{equation: fibre metric} 
    \|w\|_x=\|\hat\rho(g^{-1})w\|_o
\end{equation}
We check that this metric is well-defined. Namely, for $go=g'o=x$, $h:=(g')^{-1}g\in \SU(2)$, we have 
$$\|w\|_x=\|\hat\rho(g'^{-1})w\|_o=\|\hat\rho(h)\hat\rho(g^{-1})w\|_o=\|\hat\rho(g^{-1})w\|_o$$
(recall that this norm coincides with Equation (\ref{equation: matrix inner product}), which is invariant under $\SU(2)$).
This metric is also automatically equivariant under the $\Gamma$ action.

Let us denote by $\pi_2:N\to M'$, $\pi_3:M'\to M$ the projections of the bundles, and $\pi_4:(\widetilde M\times F)\times \T^d\to N=\Gamma\backslash((\widetilde M\times F)\times \T^d)$ and $\pi_5: \widetilde M\times F\to M'=\Gamma\backslash(\widetilde M\times F)$ the projections induced by the quotients. Recall that $\pi_0:\widetilde M\times F\to \widetilde M$ is the projection to the first coordinate.

We now prove that $\Phi^t: N\to N$ that we constructed at the beginning of Section \ref{section: proof of the main theorem}, satisfying all the assumptions of Theorem \ref{theorem: main theorem}, is an Anosov flow on $N$.

Recall one of the definitions of the Busemann function and some of its properties (see for example \cite[Chapter 2]{ballmann}). We let $b:\partial\widetilde M\times \widetilde M\times \widetilde M\to \R$ be such that for $\xi\in\partial\widetilde M$ and $x,y,z\in\widetilde M$,
$$b(\xi,y,z)=\lim_{x\to\xi} (d(x,y)-d(x,z)),$$
which measures the distance between the horosphere containing $y$ and the horosphere containing $z$.
We can check that 
\begin{equation} \label{equation: lipschitz}
    |b(\xi,y,z)-b(\xi,y,z')| = |b(\xi,z',z)| = \lim_{x\to\xi}|d(x,z')-d(x,z)| \leq d(z,z').
\end{equation}

We have the following elementary lemma for the configuration in Figure \ref{figure: estimates}.

\begin{lemma} \label{lemma: elementary lemma}
    Suppose $oB$ is part of a quasi-geodesic that satisfies for some $k>0$,
    $$\frac{1}{k}|oB|-k\leq d(o,B) \leq k|oB|+k $$
    where $|oB|$ denotes the distance between $o$ and $B$ along the quasi-geodesic. Let $o\xi^+$ denote a geodesic that passes through $o$ and shares the same forward endpoint $\xi^+$ with $oB$.
    In addition, assume that the part of the quasi-geodesic between $o$ and $\xi^+$ stays within a $K$ neighborhood of $o\xi^+$ for some $K>0$.
    Then we have 
    $$d(o,A)\geq \frac{1}{k}|oB|-k-2K.$$
\end{lemma}
\begin{proof}
    Draw a geodesic $BD$ that is perpendicular to $o\xi^+$ and let $D$ denote the foot. We have 
    $$\frac{1}{k}|oB|-k\leq d(o,B)\leq d(o,D)+d(B,D)\leq d(o,A)+d(A,D)+d(B,D).$$
    Using Equation (\ref{equation: lipschitz}),
    $$d(A,D)\leq d(B,D)\leq K.$$
    Therefore $\frac{1}{k}|oB|-k\leq d(o,A)+2K$.
\end{proof}

\begin{figure}[!h]
\centering
\begin{tikzpicture}[
  scale=1,
  line cap=round,
  line join=round,
  geodesic/.style={thick},
  horosphere/.style={densely dashed,thin},
  dot/.style={circle,fill,inner sep=1.5pt}
]
  \draw[thick] (0,0) circle (3);
  \node[below right] at (2.12,-2.12)
    {$\partial\widetilde M$};

  \pgfmathsetmacro{\Bx}{-8*1.3^2/(16+1.3^2)}
  \pgfmathsetmacro{\By}{3-32*1.3/(16+1.3^2)}

  \coordinate (o)    at (0,0);
  \coordinate (ot)   at (0,.4);
  \coordinate (q)    at (-.98630137,.36986301);
  \coordinate (qt)   at (\Bx,\By);
  \coordinate (xi)   at (0,3);

  \pgfmathsetmacro{\cy}{(\Bx*\Bx+\By*\By+9)/(2*\By)}
  \pgfmathsetmacro{\rr}{sqrt(\cy*\cy-9)}
  \pgfmathsetmacro{\fy}{\cy-\rr}
  \pgfmathsetmacro{\endangle}{-90-atan(-\Bx/(\cy-\By))}
  \coordinate (foot) at (0,\fy);

  \draw[horosphere] (0,1.5) circle (1.5);
  \draw[horosphere] (0,1.7) circle (1.3);

  \pgfmathsetmacro{\hr}{(3-\fy)/2}
  \pgfmathsetmacro{\hy}{(3+\fy)/2}
  \draw[horosphere] (0,\hy) circle[radius=\hr];

  \draw[geodesic] (0,-3) -- (xi);
  \draw[geodesic] (-2.88,-.84)
    arc[start angle=-73.739795,end angle=0,radius=4];

  \draw[geodesic] (foot)
    arc[start angle=-90,end angle=\endangle,radius=\rr];

  \draw[magenta!65!pink,thick]
    (o)
    .. controls (-1.10,.05)
       and ({\Bx-.35},{\By-.25}) .. (qt)
    .. controls ({\Bx+.35},{\By+.25})
       and (.90,1.05) .. (.65,1.55)
    .. controls (.40,2.05)
       and (-.45,1.95) .. (-.30,2.40)
    .. controls (-.15,2.85)
       and (-.05,2.85) .. (xi);

  \coordinate (ximinus) at (1.8,-2.4);

  \draw[magenta!65!pink,thick]
    (ximinus)
    .. controls (1.55,-2.22) and (1.02,-2.05) .. (1.15,-1.70)
    .. controls (1.28,-1.35) and (1.62,-1.15) .. (1.30,-.85)
    .. controls (.98,-.55) and (1.10,-.05) .. (o);

  \node[dot]
    at (ximinus) {};

  \node[dot,label={below right:$o$}] at (o) {};
  \node[dot,label={[font=\small,label distance=-2pt,inner sep=1pt]below right:$A$}] at (ot) {};
  \node[dot,label={[font=\small,label distance=-2pt,inner sep=1pt]above left:$C$}] at (q) {};
  \node[dot,label={[font=\small,label distance=-2pt,inner sep=1pt]above left:$B$}] at (qt) {};
  \node[dot,label={[font=\small,label distance=-2pt,inner sep=1pt]above right:$D$}] at (foot) {};
  \node[dot,label={above right:$\xi^+(x_o)$}] at (xi) {};
\end{tikzpicture}

\caption{A quasi-geodesic orbit (in pink), geodesics, and horospheres.}
\label{figure: estimates}
\end{figure}

\begin{proof}[Proof of Theorem \ref{theorem: main theorem}]
    A fibrewise Anosov flow over an Anosov flow is an Anosov flow on the total space (\cite[Proposition 2.4]{BBGRH}). We only need to find the stable and unstable distributions and prove the uniform hyperbolicity along the fibres.
    
    Consider a point $(x,y,z)\in (\widetilde M\times F)\times \T^d$. Because the metric in the $\T^d$ fibres are defined only depending on the point $x\in\widetilde M$, we consider the projected orbit $\pi_0\tilde\varphi^t(x,y)$. 
    Every projected orbit $\pi_0\tilde\varphi^t(x,y)$ stays in a $K$ neighborhood of a geodesic of $\widetilde M$ for some uniform $K>0$ (see, for example, the Morse Lemma, \cite{bridson-haefliger} III.H Theorem 1.7.), and hence is freely homotopic to this geodesic and shares the same endpoints on $\partial\widetilde M$.
    Let us associate $(x,y)$ with the endpoints of $\pi_0\tilde\varphi^t(x,y)$. We define maps $\sigma^\pm:\widetilde M\times F\to \partial\widetilde M$,
    $$\sigma^\pm (x,y)=\lim_{t\to\pm\infty} \pi_0\tilde\varphi^t(x,y).$$
    They are continuous in $(x,y)$, and we can check that for any $\gamma\in \Gamma$,
    $$\gamma.\sigma^\pm(x,y)=\gamma.\lim_{t\to\pm\infty}\pi_0\tilde\varphi^t(x,y)=\lim_{t\to\pm\infty}\gamma.\pi_0\tilde\varphi^t(x,y)=\lim_{t\to\pm\infty}\pi_0\tilde\varphi^t(\gamma.(x,y))=\sigma^\pm(\gamma.(x,y)),$$
    again because $\pi_0\tilde\varphi^t(x,y)$ is a quasi-geodesic, $\gamma$ passes into the limit.
    
    Recall the definition of the boundary maps in Equation (\ref{equation: boundary maps}), $\xi^\pm:\partial\widetilde M\to\mathrm{Gr}_4(\R)$. We let $V^s_{(x,y)}=\xi^+\circ\sigma^+(x,y)$ and $V^u_{(x,y)}=\xi^-\circ\sigma^-(x,y)$. We have $\rho(\gamma)V^{s/u}_{(x,y)}=V^{s/u}_{\gamma.(x,y)}$, because $\sigma^{\pm}$ and $\xi^\pm$ are equivariant under the $\Gamma$ action. Note that $V^{s/u}_{(x,y)}$ is not necessarily constant along $F$.

    Without loss of generality, up to isometries, we can assume that our quasi-geodesic orbit $\pi_0\tilde\varphi^t(x,y)$ and a geodesic intersecting it nontrivially at a point and sharing the forward endpoint $\xi^+$ on $\partial\widetilde M$ (not necessarily sharing $\xi^-$) are in the configuration of Figure \ref{figure: estimates} and satisfy the assumptions in Lemma \ref{lemma: elementary lemma}. Assume here that the forward endpoint is $\xi^+(x_o)$ as in the last section.
    Also note that $V^{s/u}_{(o,y)}=E^{+/-}$ if $\pi_0\tilde\varphi^t(o,y)$ is the standard axis, and any vectors in $V^{s/u}_{(o,y)}$ are contracted and expanded at rates $e^{-t/2}$ or $e^{t/2}$ respectively.
    
    Now, take any vector $v\in V^s_{(x,y,z)}$, where $z$ can be any point in the torus fibre $\T^d_{(x,y)}$ over $(x,y)$ so we can omit below. Also, the way in which we defined the metric does not depend on $y$. We use Lemma \ref{lemma: elementary lemma}, in the case where $\pi_0\tilde\varphi^t=:B$, so we have $|oB|=t$. We can pick $g=h\begin{pmatrix}
        e^{d(o,A)/2} & \\ & e^{-d(o,A)/2}
    \end{pmatrix}\in \SL(2,\C)$ such that $g.o=B$, where $h\in P^+$ stabilizing $\xi^+(x_o)$ preserves each horosphere centered at $\xi^+(x_o)$ and is a translation along each horosphere. Here $h$ has the form $\begin{pmatrix}
        1 & a \\ 0 & 1
    \end{pmatrix}$ for some $a\in\C$ and is an isometry on $E^+$.
    We then have
    $$\|d\tilde\Phi^tv\|_{B} =\|v\|_{B}=\|\rho(g^{-1})v\|_o=e^{-d(o,A)/2}\|v\|_o \leq e^{-\frac{1}{2k}t+k/2+K}\|v\|_o.$$
    The same argument works for $V^u$ as the stable distribution for the backward orbits.
\end{proof}

This completes the proof of Theorem \ref{theorem: main theorem}. Note that the algebraic action is smooth in the $\T^d$-fibres, so the regularity of the flow depends on the base flow.

\begin{remark} \label{remark: anosov repn via GW}
    We can see that $\rho$ is an Anosov representation explicitly by checking the following. Here, let $M$ be a closed negatively curved Riemannian manifold, $N:=T^1M$ its unit tangent bundle, $\hat N:=T^1\widetilde M$, and $\tau^\pm:\hat N\to\partial\widetilde M$ the maps associating to a tangent vector the endpoints at $\pm\infty$ of the corresponding geodesic. Let $G$ be a semisimple Lie group, $(P^+,P^-)$ a pair of opposite parabolic subgroups of $G$, and $\mathcal F^\pm=G/P^\pm$.
    \begin{prop}[\cite{guichard-wienhard}, Proposition 2.7]
        Let $\rho:\pi_1(M)\to G$ be a representation. Suppose that there exist maps $\xi^\pm:\partial\widetilde M\to \mathcal F^\pm$ such that 
        \begin{enumerate}
            \item $\xi^\pm$ are continuous and $\rho$-equivariant.
            \item For all $(t^+,t^-)\in\partial\widetilde M\times \partial\widetilde M\backslash\Delta$ the pair $(\xi^+(t^+),\xi^-(t^-))$ is transverse. \begin{sloppypar}
            \item For one (and hence any) continuous and equivariant family of norms $(\|\cdot\|_{\hat n})_{\hat n\in\hat N}$ on $(T_{\xi^\pm(\tau^\pm(\hat n))}\mathcal F^\pm)_{\hat n\in\hat N}$, the following property holds: there exist positive constants $A,a$ such that for all $t\in\R^{\geq 0}$, $\hat n\in\hat N$ and $e^\pm\in T_{\xi^\pm(\tau^\pm(\hat n))}\mathcal F^\pm$, we have
            $$\|e^+\|_{\phi^{-t}\hat m}\leq Ae^{-at}\|e\|_{\hat n}, \qquad \text{and} \quad \|e^-\|_{\phi^{t}\hat m}\leq Ae^{-at}\|e\|_{\hat n}.$$ \end{sloppypar}
            Then the representation $\rho$ is $(P^+,P^-)$-Anosov.
        \end{enumerate}
    \end{prop}
    
    We have checked (1) and (2) in Section \ref{subsubsection: more linear alg}, and are left to check (3). We want to find a continuous and equivariant family of norms and show that for any vector (for simplicity we consider $\mathrm{Gr}_2(\C^4)$) $$V\in T_{\xi^+(x_o)}\mathrm{Gr}_2(\C^4)\cong\mathrm{Hom}(E^+,\C^4/E^+)\cong \mathrm{Hom}(E^+,E^-),$$
    some positive constants $A,a$ and all $t\geq 0$,
    $$\|V\|_{g_{-t}o}\leq Ae^{-at}\|V\|_o,$$
    where $g_{-t}=\begin{pmatrix}
        e^{-t/2} & \\ & e^{t/2}
    \end{pmatrix}$ along the geodesic heading towards $\xi^+(x_o)$.
    But we can simply take the usual operator norm, where we denote $x\in E^+$, $g\in \SL(2,\C)$ and the norm of $x$ is defined in the same way as in Equation (\ref{equation: fibre metric}), and therefore,
    $$\|V\|_{go}=\sup_{\|x\|_{go}=1}\|Vx\|_{go}=\sup_{\|g^{-1}x\|_o=1} \|g^{-1}Vg(g^{-1}x)\|_o=\|g^{-1}Vg\|_o.$$
    Then assuming that $E^{+/-}$ are in the given form from Section \ref{subsubsection: more linear alg} and $x\in E^+$, we have 
    \begin{align*}
        \begin{pmatrix}
        e^{t/2} & \\ & e^{-t/2}
    \end{pmatrix} V \begin{pmatrix}
        e^{-t/2} & \\ & e^{t/2}
    \end{pmatrix} \begin{pmatrix}
        a_1 & a_3 \\ 0 & 0
    \end{pmatrix} & = e^{-t/2} \begin{pmatrix}
        e^{t/2} & \\ & e^{-t/2}
    \end{pmatrix} V \begin{pmatrix}
        a_1 & a_3 \\ 0 & 0
    \end{pmatrix} \\
     &= e^{-t/2} \begin{pmatrix}
        e^{t/2} & \\ & e^{-t/2}
    \end{pmatrix} \begin{pmatrix}
        0 & 0 \\ a_2 & a_4
    \end{pmatrix} = e^{-t} V \begin{pmatrix}
        a_1 & a_3 \\ 0 & 0
    \end{pmatrix}.
    \end{align*}
    So
    $$\|V\|_{\phi^{-t}o}=\|\phi^{-t}V\phi^t\|_o=\sup_{\|x\|_o=1}\|\phi^{-t}V\phi^t x\|_o=e^{-t}\|V\|_o.$$
\end{remark}

\begin{remark} \label{remark: anosov repn via singular values}
    Our $\rho$ is more apparently $4$-Anosov in the easier definition with singular values (see for example \cite{canary}), although we would like to see the boundary maps more explicitly with the linear algebra as above.
    Note that our Claim \ref{claim: existence of representation} with the Milnor-Svarc Lemma gives an exponential gap in the singular values of $\rho(\gamma)$, $\gamma\in \Gamma$, which implies an exponential gap in the singular values of the linear flow (our $\tilde\Phi^t$) in the fibres, by some quasi-geodesic estimates (see \cite[Section 4]{bochi-potrie-sambarino}). Then \cite[Theorem A]{bochi-gourmelon} (\cite[Theorem 2.2]{bochi-potrie-sambarino}) gives the invariant distributions $E^{cs/cu}$ and dominated splittings. But in our case, we essentially have only two eigenvalues for $\gamma\in \SL(2,\C)$ (recall that the eigenvalues in the fibres have the same moduli as $\gamma$), so the dominated splitting is indeed uniformly hyperbolic.
\end{remark}

\subsection{Getting examples}

\begin{proof}[Proof of Corollary \ref{coro: over FMP}]
    The \cite{fenley-mann-potrie} example is on a flat bundle $\Gamma\backslash(\widetilde M\times\sph^1)$, where $\Gamma\backslash\widetilde M$ can be any closed hyperbolic 3-manifold. The candidate of the hyperbolic 3-manifold can be picked in the following way.
    Suppose we have a torsion free arithmetic Kleinian group that satisfies assumption (1) of Theorem \ref{theorem: main theorem}.
    Up to a finite index subgroup $\Gamma$, the corresponding hyperbolic 3-manifold admits a pseudo-Anosov flow \cite{agol}.
    Then one can construct a minimal action of $\Gamma$ on $\sph^1$ and the Cannon-Thurston map as in \cite[Section 5.2]{fenley-mann-potrie}, and then an Anosov flow on $M':=\Gamma\backslash(\widetilde M\times\sph^1)$.
    Its orbits project to geodesics on a hyperbolic 3-manifold. Also, it can be approximated by a smooth volume preserving (and hence transitive) Anosov flow by \cite{asaoka}. 
    Therefore, we can construct a fibrewise Anosov flow over it.

    The fibrewise Anosov flow is non-algebraic.
    Because $\sph^1$, $\T^d$ and hyperbolic manifolds are aspherical, the higher homotopy groups of $M'$ vanish and hence the higher homotopy groups of $N$ also vanish.
    So $N$ cannot be a sphere bundle and thus $\Phi^t$ is not a geodesic flow. We look at the short exact sequences
    $$\begin{tikzcd}
        1\rar & \Z \rar & \pi_1(M') \rar["p_1"] & \Gamma \rar & 1
    \end{tikzcd}$$
    and
    $$\begin{tikzcd}
        1\rar & \Z^d \rar & \pi_1(N) \rar["p_2"] & \pi_1(M') \rar & 1
    \end{tikzcd}.$$
    We also know that $\pi_1(N)\not\cong \Lambda\rtimes\Z$ where $\Lambda$ is a nilpotent group, because otherwise the projection $p_1\circ p_2(\pi_1(N))$ to $\Gamma$ would be solvable, but $\Gamma$ is a hyperbolic 3-manifold group.
    Hence $\Phi^t$ is not a suspension flow.
    On the other hand, $N$ is of dimension $4+8m, m\in\Z_{\geq 1}$. If $\Phi^t$ is a nil-suspension over a geodesic flow, the fibre dimension would be odd.
    The geodesic flow on the unit tangent bundle of a hyperbolic space has the property that there exists a pair of periodic orbits $a,b$, such that $a$ is freely homotopic to the reverse of $b$.
    But from \cite{BBGRH}, we know that a fibrewise Anosov flow whose base flow has this property must have even-dimensional fibres.
\end{proof}

\begin{remark} \label{remark: inf many flow inv submanifolds}
    There exist infinitely many closed flow invariant submanifolds. 

    The bundle $N$ is foliated by manifolds that lift to $\widetilde M$, and there exists a leaf that becomes a closed manifold if there exists some $z\in \T^d$, such that $\rho(\Gamma).z$ is a finite set. Then $\rho(\Gamma).z$ would be an invariant set under the $\Gamma$ action and $\Gamma\backslash((\widetilde M\times F)\times (\rho(\Gamma).z))$ would be a closed flow invariant submanifold.

    We can take the finer lattice $\frac{1}{p}\Z^d\subseteq \R^d$, where $p$ is a prime number. In a torus, it consists of finitely many points and is invariant under $\SL(d,\Z)$. Since $\frac{1}{p}\Z^d\cap \frac{1}{q}\Z^d=\{0\}$ if the primes $p\neq q$, there are infinitely many closed submanifolds that are invariant under the flow.
\end{remark}

We would like to end with some discussions on the possibility of constructing fibrewise Anosov flows over hyperbolic 3-manifolds, using the arithmetic Kleinian group actions.

To construct a fibrewise Anosov flow over an Anosov flow on a closed hyperbolic 3-manifold, we want to check whether the currently known examples satisfy all the following conditions.
\begin{enumerate}
    \item The hyperbolic 3-manifold is arithmetic, and the arithmetic Kleinian group arises from a quaternion algebra $\QA$ having no real places.
    \item The flow is quasi-geodesic.
\end{enumerate}

Besides, any torus bundle constructed in this way would not fall into any of the cases of algebraic Anosov flows. It is not a sphere bundle or a mapping torus by the same argument as in the proof of Corollary \ref{coro: over FMP}.
It is not a nil-suspension either. As before, we always have 
\begin{equation} \label{equation: sequence fibrewise over hyp 3} \begin{tikzcd}
    1\rar & \Z^d\rar["i_1"] & \pi_1(N)\rar["p_1"] & \Gamma\rar & 1
\end{tikzcd} \end{equation}
where $\Gamma$ is the fundamental group of a closed hyperbolic 3-manifold.
Assume $N$ is a nil-suspension over a geodesic flow, whose base is an $\sph^k$-bundle over another hyperbolic manifold. If $k\geq 3$, then  $\pi_k(N)\cong \pi_k(\sph^k)$, which is impossible because $\pi_k(N)$ vanishes for $k\geq 2$. On the other hand, if $k=2$, we have $\pi_3(N)\cong\pi_3(\sph^2)\cong\Z$, which we can also rule out.
If $k=1$, we have 
\begin{equation}\label{equation: sequence k=2} \begin{tikzcd}
    1\rar & \Lambda \rar["i_2"] & \pi_1(N)\rar["p_2"] & \pi_1(T^1S)\rar & 1
\end{tikzcd} \end{equation}
for some closed hyperbolic surface $S$ and nilpotent group $\Lambda$.
We consider $\Lambda':=p_1\circ i_2(\Lambda)$ in $\Gamma$. As the centralizer $C_\Gamma(\gamma)\cong\Z$ for every $\gamma\in\Gamma$, if we take $\gamma\in Z(\Lambda')$ where $Z(\Lambda')$ denotes the center, $\Lambda' < C_\Gamma(\gamma)\cong\Z$
, so $\Lambda'$ is either trivial or isomorphic to $\Z$. But $\Lambda'$ is also a normal subgroup of $\Gamma$, by \cite[Proposition III.$\Gamma$.3.16]{bridson-haefliger} $\Gamma$ would be virtually cyclic, so $\Lambda'$ must be trivial and hence $i_2(\Lambda)<\ker p_1=i_1(\Z^d)$. Because the nilmanifold fibre should have dimension $d$ too, and its fundamental group is abelian, we can conclude that $N$ is also a $\T^d$-bundle over $T^1S$, which is impossible.

Anosov flows on hyperbolic 3-manifolds are transitive \cite{brunella}.
Fenley \cite{fenley-AnosovFlowsIn3Manifolds, fenley-QGImpliesNonRcov, fenley-NonRcovImpliesQG} showed that an Anosov flow on a closed hyperbolic 3-manifold is quasi-geodesic if and only if it is non-$\R$-covered. While it might be easier to determine whether the Goodman surgery \cite{goodman} on a suspension flow produces an arithmetic manifold, such examples are $\R$-covered and hence not quasi-geodesic. Moreover, as they have infinitely many periodic orbits in the same homotopy class, they cannot be the base of a fibrewise Anosov flow by \cite{PZ26}.
Whether other candidates for Anosov flows on hyperbolic 3-manifolds, e.g. \cite{bonatti-iakovoglou, beguin-yu, BBMY}, are arithmetic is itself a hard question.
Bin Yu, Biao Ma and Chi Cheuk Tsang, in recent private conversations,  mentioned that it is possible to produce arithmetic hyperbolic 3-manifolds that support Anosov flows, which could make our fibrewise construction possible, but we would not pursue this question further here.

\

\bibliographystyle{alpha}
\bibliography{bibfile}

@book {maclachlan-reid,
    AUTHOR = {Maclachlan, Colin and Reid, Alan W.},
     TITLE = {The arithmetic of hyperbolic 3-manifolds},
    SERIES = {Graduate Texts in Mathematics},
    VOLUME = {219},
 PUBLISHER = {Springer-Verlag, New York},
      YEAR = {2003},
     PAGES = {xiv+463},
      ISBN = {0-387-98386-4},
   MRCLASS = {57M50 (11R52)},
  MRNUMBER = {1937957},
MRREVIEWER = {Kerry\ N.\ Jones},
       DOI = {10.1007/978-1-4757-6720-9},
       URL = {https://doi.org/10.1007/978-1-4757-6720-9},
}

@book {rosen,
    AUTHOR = {Rosen, Kenneth H.},
     TITLE = {Elementary number theory and its applications},
   EDITION = {Fourth},
 PUBLISHER = {Addison-Wesley, Reading, MA},
      YEAR = {2000},
     PAGES = {xviii+638},
      ISBN = {0-201-87073-8},
   MRCLASS = {11-01},
  MRNUMBER = {1739433},
MRREVIEWER = {Norman\ J.\ Richert},
}

@incollection {tomter,
    AUTHOR = {Tomter, Per},
     TITLE = {Anosov flows on infra-homogeneous spaces},
 BOOKTITLE = {Global {A}nalysis ({P}roc. {S}ympos. {P}ure {M}ath., {V}ols.
              {XIV}, {XV}, {XVI}, {B}erkeley, {C}alif., 1968)},
    SERIES = {Proc. Sympos. Pure Math.},
    VOLUME = {XIV-XVI},
     PAGES = {299--327},
 PUBLISHER = {Amer. Math. Soc., Providence, RI},
      YEAR = {1970},
   MRCLASS = {57.48 (34.00)},
  MRNUMBER = {279831},
MRREVIEWER = {T.\ Nagano},
}

@article {fenley-QGImpliesNonRcov,
    AUTHOR = {Fenley, S\'ergio R.},
     TITLE = {Quasigeodesic {A}nosov flows and homotopic properties of flow
              lines},
   JOURNAL = {J. Differential Geom.},
  FJOURNAL = {Journal of Differential Geometry},
    VOLUME = {41},
      YEAR = {1995},
    NUMBER = {2},
     PAGES = {479--514},
      ISSN = {0022-040X,1945-743X},
   MRCLASS = {58F15 (57M50 57R30 58F18)},
  MRNUMBER = {1331975},
MRREVIEWER = {Lee\ Mosher},
       URL = {http://projecteuclid.org/euclid.jdg/1214456224},
}

@article {fenley-NonRcovImpliesQG,
    AUTHOR = {Fenley, Sergio R.},
     TITLE = {Non {$\Bbb R$}-covered {A}nosov flows in hyperbolic
              3-manifolds are quasigeodesic},
   JOURNAL = {Geom. Funct. Anal.},
  FJOURNAL = {Geometric and Functional Analysis},
    VOLUME = {36},
      YEAR = {2026},
    NUMBER = {2},
     PAGES = {412--508},
      ISSN = {1016-443X,1420-8970},
   MRCLASS = {57R30 (37C85 37D20 37E10 53C12)},
  MRNUMBER = {5048793},
       DOI = {10.1007/s00039-026-00733-5},
       URL = {https://doi.org/10.1007/s00039-026-00733-5},
}

@article {fenley-AnosovFlowsIn3Manifolds,
    AUTHOR = {Fenley, S\'ergio R.},
     TITLE = {Anosov flows in {$3$}-manifolds},
   JOURNAL = {Ann. of Math. (2)},
  FJOURNAL = {Annals of Mathematics. Second Series},
    VOLUME = {139},
      YEAR = {1994},
    NUMBER = {1},
     PAGES = {79--115},
      ISSN = {0003-486X,1939-8980},
   MRCLASS = {58F15 (57M50 57N10 58F18)},
  MRNUMBER = {1259365},
MRREVIEWER = {Lee\ Mosher},
       DOI = {10.2307/2946628},
       URL = {https://doi.org/10.2307/2946628},
}

@article {bonatti-iakovoglou,
    AUTHOR = {Bonatti, Christian and Iakovoglou, Ioannis},
     TITLE = {Anosov flows on 3-manifolds: the surgeries and the foliations},
   JOURNAL = {Ergodic Theory Dynam. Systems},
  FJOURNAL = {Ergodic Theory and Dynamical Systems},
    VOLUME = {43},
      YEAR = {2023},
    NUMBER = {4},
     PAGES = {1129--1188},
      ISSN = {0143-3857,1469-4417},
   MRCLASS = {37D20 (37D40 57M10 57R30)},
  MRNUMBER = {4555824},
       DOI = {10.1017/etds.2021.170},
       URL = {https://doi.org/10.1017/etds.2021.170},
}

@incollection {goodman,
    AUTHOR = {Goodman, Sue},
     TITLE = {Dehn surgery on {A}nosov flows},
 BOOKTITLE = {Geometric dynamics ({R}io de {J}aneiro, 1981)},
    SERIES = {Lecture Notes in Math.},
    VOLUME = {1007},
     PAGES = {300--307},
 PUBLISHER = {Springer, Berlin},
      YEAR = {1983},
      ISBN = {3-540-12336-9},
   MRCLASS = {58F15 (57R65)},
  MRNUMBER = {1691596},
       DOI = {10.1007/BFb0061421},
       URL = {https://doi.org/10.1007/BFb0061421},
}

@book {bridson-haefliger,
    AUTHOR = {Bridson, Martin R. and Haefliger, Andr\'e},
     TITLE = {Metric spaces of non-positive curvature},
    SERIES = {Grundlehren der mathematischen Wissenschaften [Fundamental
              Principles of Mathematical Sciences]},
    VOLUME = {319},
 PUBLISHER = {Springer-Verlag, Berlin},
      YEAR = {1999},
     PAGES = {xxii+643},
      ISBN = {3-540-64324-9},
   MRCLASS = {53C23 (20F65 53C70 57M07)},
  MRNUMBER = {1744486},
MRREVIEWER = {Athanase\ Papadopoulos},
       DOI = {10.1007/978-3-662-12494-9},
       URL = {https://doi.org/10.1007/978-3-662-12494-9},
}

@incollection {parker,
    AUTHOR = {Parker, John R.},
     TITLE = {Traces in complex hyperbolic geometry},
 BOOKTITLE = {Geometry, topology and dynamics of character varieties},
    SERIES = {Lect. Notes Ser. Inst. Math. Sci. Natl. Univ. Singap.},
    VOLUME = {23},
     PAGES = {191--245},
 PUBLISHER = {World Sci. Publ., Hackensack, NJ},
      YEAR = {2012},
      ISBN = {978-981-4401-35-7; 981-4401-35-8},
   MRCLASS = {32Q45 (51M10 57S25)},
  MRNUMBER = {2987619},
MRREVIEWER = {Krishnendu\ Gongopadhyay},
       DOI = {10.1142/9789814401364\_0006},
       URL = {https://doi.org/10.1142/9789814401364_0006},
}

@article {asaoka,
    AUTHOR = {Asaoka, Masayuki},
     TITLE = {On invariant volumes of codimension-one {A}nosov flows and the
              {V}erjovsky conjecture},
   JOURNAL = {Invent. Math.},
  FJOURNAL = {Inventiones Mathematicae},
    VOLUME = {174},
      YEAR = {2008},
    NUMBER = {2},
     PAGES = {435--462},
      ISSN = {0020-9910,1432-1297},
   MRCLASS = {37D20 (37C10)},
  MRNUMBER = {2439611},
MRREVIEWER = {Boris\ Hasselblatt},
       DOI = {10.1007/s00222-008-0151-9},
       URL = {https://doi.org/10.1007/s00222-008-0151-9},
}

@article {BBGRH,
    AUTHOR = {Barthelm\'e, Thomas and Bonatti, Christian and Gogolev, Andrey
              and Rodriguez Hertz, Federico},
     TITLE = {Anomalous {A}nosov flows revisited},
   JOURNAL = {Proc. Lond. Math. Soc. (3)},
  FJOURNAL = {Proceedings of the London Mathematical Society. Third Series},
    VOLUME = {122},
      YEAR = {2021},
    NUMBER = {1},
     PAGES = {93--117},
      ISSN = {0024-6115,1460-244X},
   MRCLASS = {37E35 (37D20 57K35)},
  MRNUMBER = {4210258},
MRREVIEWER = {Norikazu\ Hashiguchi},
       DOI = {10.1112/plms.12321},
       URL = {https://doi.org/10.1112/plms.12321},
}

@article {brunella,
    AUTHOR = {Brunella, Marco},
     TITLE = {Separating the basic sets of a nontransitive {A}nosov flow},
   JOURNAL = {Bull. London Math. Soc.},
  FJOURNAL = {The Bulletin of the London Mathematical Society},
    VOLUME = {25},
      YEAR = {1993},
    NUMBER = {5},
     PAGES = {487--490},
      ISSN = {0024-6093,1469-2120},
   MRCLASS = {58F15},
  MRNUMBER = {1233413},
MRREVIEWER = {Lawrence\ Conlon},
       DOI = {10.1112/blms/25.5.487},
       URL = {https://doi.org/10.1112/blms/25.5.487},
}

@article {handel-thurston,
    AUTHOR = {Handel, Michael and Thurston, William P.},
     TITLE = {Anosov flows on new three manifolds},
   JOURNAL = {Invent. Math.},
  FJOURNAL = {Inventiones Mathematicae},
    VOLUME = {59},
      YEAR = {1980},
    NUMBER = {2},
     PAGES = {95--103},
      ISSN = {0020-9910,1432-1297},
   MRCLASS = {58F15},
  MRNUMBER = {577356},
MRREVIEWER = {Zbigniew\ Nitecki},
       DOI = {10.1007/BF01390039},
       URL = {https://doi.org/10.1007/BF01390039},
}

@article {bonatti-langevin,
    AUTHOR = {Bonatti, Christian and Langevin, R\'emi},
     TITLE = {Un exemple de flot d'{A}nosov transitif transverse \`a{} un
              tore et non conjugu\'e{} \`a{} une suspension},
   JOURNAL = {Ergodic Theory Dynam. Systems},
  FJOURNAL = {Ergodic Theory and Dynamical Systems},
    VOLUME = {14},
      YEAR = {1994},
    NUMBER = {4},
     PAGES = {633--643},
      ISSN = {0143-3857,1469-4417},
   MRCLASS = {58F15},
  MRNUMBER = {1304136},
MRREVIEWER = {Boris\ Hasselblatt},
       DOI = {10.1017/S0143385700008099},
       URL = {https://doi.org/10.1017/S0143385700008099},
}

@incollection {franks-williams,
    AUTHOR = {Franks, John and Williams, Bob},
     TITLE = {Anomalous {A}nosov flows},
 BOOKTITLE = {Global theory of dynamical systems ({P}roc. {I}nternat.
              {C}onf., {N}orthwestern {U}niv., {E}vanston, {I}ll., 1979)},
    SERIES = {Lecture Notes in Math.},
    VOLUME = {819},
     PAGES = {158--174},
 PUBLISHER = {Springer, Berlin},
      YEAR = {1980},
      ISBN = {3-540-10236-1},
   MRCLASS = {58F15},
  MRNUMBER = {591182},
MRREVIEWER = {D.\ K.\ Arrowsmith},
}

@article {BBY,
    AUTHOR = {B\'eguin, François and Bonatti, Christian and Yu, Bin},
     TITLE = {Building {A}nosov flows on 3-manifolds},
   JOURNAL = {Geom. Topol.},
  FJOURNAL = {Geometry \& Topology},
    VOLUME = {21},
      YEAR = {2017},
    NUMBER = {3},
     PAGES = {1837--1930},
      ISSN = {1465-3060,1364-0380},
   MRCLASS = {37D20 (57M50)},
  MRNUMBER = {3650083},
MRREVIEWER = {Rafael\ Oswaldo\ Ruggiero},
       DOI = {10.2140/gt.2017.21.1837},
       URL = {https://doi.org/10.2140/gt.2017.21.1837},
}

@article {neige,
    AUTHOR = {Paulet, Neige},
     TITLE = {Anosov flows in dimension 3 from gluing building blocks with
              quasi-transverse boundary},
   JOURNAL = {J. Mod. Dyn.},
  FJOURNAL = {Journal of Modern Dynamics},
    VOLUME = {21},
      YEAR = {2025},
     PAGES = {21--240},
      ISSN = {1930-5311,1930-532X},
   MRCLASS = {37D20 (37C10 37D05 57K30)},
  MRNUMBER = {4879873},
MRREVIEWER = {Thilo\ Kuessner},
       DOI = {10.3934/jmd.2025002},
       URL = {https://doi.org/10.3934/jmd.2025002},
}

@article {bowden-mann,
    AUTHOR = {Bowden, Jonathan and Mann, Kathryn},
     TITLE = {{$C^0$} stability of boundary actions and inequivalent
              {A}nosov flows},
   JOURNAL = {Ann. Sci. \'Ec. Norm. Sup\'er. (4)},
  FJOURNAL = {Annales Scientifiques de l'\'Ecole Normale Sup\'erieure.
              Quatri\`eme S\'erie},
    VOLUME = {55},
      YEAR = {2022},
    NUMBER = {4},
     PAGES = {1003--1046},
      ISSN = {0012-9593,1873-2151},
   MRCLASS = {53C22 (37C15 37C85 37D20)},
  MRNUMBER = {4468857},
MRREVIEWER = {Boris\ Hasselblatt},
       DOI = {10.24033/asens.2512},
       URL = {https://doi.org/10.24033/asens.2512},
}

@article {barbot-maquera,
    AUTHOR = {Barbot, Thierry and Maquera, Carlos},
     TITLE = {Algebraic {A}nosov actions of nilpotent {L}ie groups},
   JOURNAL = {Topology Appl.},
  FJOURNAL = {Topology and its Applications},
    VOLUME = {160},
      YEAR = {2013},
    NUMBER = {1},
     PAGES = {199--219},
      ISSN = {0166-8641,1879-3207},
   MRCLASS = {37D20 (22E25)},
  MRNUMBER = {2995092},
MRREVIEWER = {Jinpeng\ An},
       DOI = {10.1016/j.topol.2012.10.012},
       URL = {https://doi.org/10.1016/j.topol.2012.10.012},
}

@article {agol,
    AUTHOR = {Agol, Ian},
     TITLE = {The virtual {H}aken conjecture},
      NOTE = {With an appendix by Agol, Daniel Groves, and Jason Manning},
   JOURNAL = {Doc. Math.},
  FJOURNAL = {Documenta Mathematica},
    VOLUME = {18},
      YEAR = {2013},
     PAGES = {1045--1087},
      ISSN = {1431-0635,1431-0643},
   MRCLASS = {20F67 (57Mxx)},
  MRNUMBER = {3104553},
MRREVIEWER = {Thomas\ Koberda},
       URL = {https://elibm.org/article/10000267},
}

@article {zeghib,
    AUTHOR = {Zeghib, A.},
     TITLE = {Sur une notion d'autonomie de syst\`emes dynamiques,
              appliqu\'ee aux ensembles invariants des flots d'{A}nosov
              alg\'ebriques},
   JOURNAL = {Ergodic Theory Dynam. Systems},
  FJOURNAL = {Ergodic Theory and Dynamical Systems},
    VOLUME = {15},
      YEAR = {1995},
    NUMBER = {1},
     PAGES = {175--207},
      ISSN = {0143-3857,1469-4417},
   MRCLASS = {58F15},
  MRNUMBER = {1314975},
MRREVIEWER = {Boris\ Hasselblatt},
       DOI = {10.1017/S0143385700008300},
       URL = {https://doi.org/10.1017/S0143385700008300},
}

@article {guichard-wienhard,
    AUTHOR = {Guichard, Olivier and Wienhard, Anna},
     TITLE = {Anosov representations: domains of discontinuity and
              applications},
   JOURNAL = {Invent. Math.},
  FJOURNAL = {Inventiones Mathematicae},
    VOLUME = {190},
      YEAR = {2012},
    NUMBER = {2},
     PAGES = {357--438},
      ISSN = {0020-9910,1432-1297},
   MRCLASS = {22F30 (32G15 53C30 53D25)},
  MRNUMBER = {2981818},
MRREVIEWER = {Pablo\ Su\'arez-Serrato},
       DOI = {10.1007/s00222-012-0382-7},
       URL = {https://doi.org/10.1007/s00222-012-0382-7},
}

@article {bochi-potrie-sambarino,
    AUTHOR = {Bochi, Jairo and Potrie, Rafael and Sambarino, Andr\'es},
     TITLE = {Anosov representations and dominated splittings},
   JOURNAL = {J. Eur. Math. Soc. (JEMS)},
  FJOURNAL = {Journal of the European Mathematical Society (JEMS)},
    VOLUME = {21},
      YEAR = {2019},
    NUMBER = {11},
     PAGES = {3343--3414},
      ISSN = {1435-9855,1435-9863},
   MRCLASS = {22E40 (20F67 37B99 37D30 53C35)},
  MRNUMBER = {4012341},
MRREVIEWER = {Alejandro\ Ucan-Puc},
       DOI = {10.4171/JEMS/905},
       URL = {https://doi.org/10.4171/JEMS/905},
}

@book {ballmann,
    AUTHOR = {Ballmann, Werner},
     TITLE = {Lectures on spaces of nonpositive curvature},
    SERIES = {DMV Seminar},
    VOLUME = {25},
      NOTE = {With an appendix by Misha Brin},
 PUBLISHER = {Birkh\"auser Verlag, Basel},
      YEAR = {1995},
     PAGES = {viii+112},
      ISBN = {3-7643-5242-6},
   MRCLASS = {53C21 (58F17)},
  MRNUMBER = {1377265},
MRREVIEWER = {Boris\ Hasselblatt},
       DOI = {10.1007/978-3-0348-9240-7},
       URL = {https://doi.org/10.1007/978-3-0348-9240-7},
}

@article {bochi-gourmelon,
    AUTHOR = {Bochi, Jairo and Gourmelon, Nicolas},
     TITLE = {Some characterizations of domination},
   JOURNAL = {Math. Z.},
  FJOURNAL = {Mathematische Zeitschrift},
    VOLUME = {263},
      YEAR = {2009},
    NUMBER = {1},
     PAGES = {221--231},
      ISSN = {0025-5874,1432-1823},
   MRCLASS = {37D30},
  MRNUMBER = {2529495},
MRREVIEWER = {Alexander\ Arbieto},
       DOI = {10.1007/s00209-009-0494-y},
       URL = {https://doi.org/10.1007/s00209-009-0494-y},
}

@misc{fenley-mann-potrie,
      title={Exotic codimension one Anosov flows}, 
      author={Sergio Fenley and Kathryn Mann and Rafael Potrie},
      year={2026},
      note={arXiv:2605.25082},
      archivePrefix={arXiv},
      primaryClass={math.DS},
      url={https://arxiv.org/abs/2605.25082}, 
}

@misc{beguin-yu,
      title={Existence of arbitrary large numbers of non-$\mathbb{R}$-covered Anosov flows on hyperbolic $3$-manifolds}, 
      author={François Béguin and Bin Yu},
      year={2024},
      note={arXiv:2402.06551},
      archivePrefix={arXiv},
      primaryClass={math.DS},
      url={https://arxiv.org/abs/2402.06551}, 
}

@misc{BBMY,
      title={Construction of Anosov flows on fibered hyperbolic 3-manifolds}, 
      author={François Béguin and Christian Bonatti and Biao Ma and Bin Yu},
      year={2026},
      note={arXiv:2603.06105},
      archivePrefix={arXiv},
      primaryClass={math.DS},
      url={https://arxiv.org/abs/2603.06105}, 
}

@misc{PZ26,
      title={An obstruction to fiberwise Anosov flows over 3-dimensional Anosov flows}, 
      author={Neige Paulet and Danyu Zhang},
      year={2026},
      note={arXiv:2601.18487},
      archivePrefix={arXiv},
      primaryClass={math.DS},
      url={https://arxiv.org/abs/2601.18487}, 
}

@misc{filip-fisher-lowe,
      title={Finiteness of totally geodesic hypersurfaces}, 
      author={Simion Filip and David Fisher and Ben Lowe},
      year={2025},
      note={arXiv:2408.03430},
      archivePrefix={arXiv},
      primaryClass={math.DG},
      url={https://arxiv.org/abs/2408.03430}, 
}

@misc{canary,
      title={ANOSOV REPRESENTATIONS: INFORMAL LECTURE NOTES},
      author={Richard D. Canary},
      year={2025},
      note={\url{https://websites.umich.edu/~canary/lecnotespublic61125.pdf}},
}

\end{document}